\documentclass[12pt,a4paper]{article}%
\usepackage[utf8]{inputenc}
\usepackage{amsmath}
\usepackage{amsfonts}
\usepackage{amssymb}
\usepackage{graphicx}
\usepackage[space]{grffile}
\usepackage{hyperref}
\usepackage{layout}
\usepackage{longtable}%
\providecommand{\U}[1]{\protect\rule{.1in}{.1in}}
\newtheorem{theorem}{Theorem}
\newtheorem{theorem*}{Example}

\newtheorem{conjecture}[theorem]{Conjecture}
\newtheorem{corollary}[theorem]{Corollary}

\newtheorem{lemma}[theorem]{Lemma}

\newtheorem{problem}[theorem]{Problem}
\newtheorem{proposition}[theorem]{Proposition}

\newenvironment{proof}[1][Proof]{\noindent\textbf{#1.} }{\ \hfill \rule{0.5em}{0.5em}\bigskip}
\graphicspath{{D:/Dropbox/Riste-Sedlar/Edge colorings/Paper4/Slike/}}
\begin{document}

\title{Proper $\mathbb{Z}_{4}\times\mathbb{Z}_{2}$-colorings: structural
characterization with application to some snarks}
\author{Jelena Sedlar$^{1,3}$,\\Riste \v Skrekovski$^{2,3}$ \\[0.3cm] {\small $^{1}$ \textit{University of Split, Faculty of civil
engineering, architecture and geodesy, Croatia}}\\[0.1cm] {\small $^{2}$ \textit{University of Ljubljana, FMF, 1000 Ljubljana,
Slovenia }}\\[0.1cm] {\small $^{3}$ \textit{Faculty of Information Studies, 8000 Novo
Mesto, Slovenia }}\\[0.1cm] }
\maketitle

\begin{abstract}
A proper abelian coloring of a cubic graph $G$ by a finite abelian group $A$
is any proper edge-coloring of $G$ by the non-zero elements of $A$ such that
the sum of the colors of the three edges incident to any vertex $v$ of $G$
equals zero. It is known that cyclic groups of order smaller than 10 do not
color all bridgeless cubic graphs, and that all abelian groups of order at
least $12$ do. This leaves the question open for the four so called
exceptional groups $\mathbb{Z}_{4}\times\mathbb{Z}_{2}$, $\mathbb{Z}_{3}%
\times\mathbb{Z}_{3}$, $\mathbb{Z}_{10}$ and $\mathbb{Z}_{11}$ for snarks. It
is conjectured in literature that every cubic graph has a proper abelian
coloring by each exceptional group and it is further known that the existence
of a proper $\mathbb{Z}_{4}\times\mathbb{Z}_{2}$-coloring of $G$ implies the
existence of a proper coloring of $G$ by all the remaining exceptional groups.
In this paper, we give a characterization of a proper $\mathbb{Z}_{4}%
\times\mathbb{Z}_{2}$-coloring in terms of the existence of a matching $M$ in
a $2$-factor $F$ of $G$ with particular properties. Moreover, in order to
modify an arbitrary matching $M$ so that it meets the requirements of the
characterization, we first introduce an incidence structure of the cycles of
$F$ in relation to the cycles of $G-M.$ Further, we provide a sufficient
condition under which $M$ can be modified into a desired matching in terms of
particular properties of the introduced incidence structure. We conclude the
paper by applying the results to some oddness two snarks, in particular to
permutation snarks. We believe that the approach of this paper with some
additional refinements extends to larger classes of snarks, if not to all in general.

\end{abstract}



\section{Introduction}

Let $G$ be a bridgeless cubic graph and $A$ an abelian group. An \emph{abelian
coloring} of the graph $G$ by the group $A$, or an $A$\emph{-coloring} for
short, is any edge-coloring of $G$ by the non-zero elements of $A$ such that
the sum of the colors of the three edges incident to any vertex $v\in V(G)$
equals zero. Notice that an abelian coloring does not have to be proper. An
abelian coloring is a generalization of 3-edge-colorings of cubic graphs
proposed by Archdeacon \cite{Archdeacon1986}, wherein it was conjectured that
an $A$-coloring exists for each bridgeless cubic graph and each abelian group
$A$ of order at least five. This conjecture was resolved in positive by
M\'{a}\v{c}ajov\'{a} et al. \cite{Macajova2005} and the proof relies on the
fact that abelian coloring does not have to be proper.

The conjecture of Archdeacon \cite{Archdeacon1986} left out the group
$\mathbb{Z}_{2}\times\mathbb{Z}_{2},$ indicating that coloring graphs by this
group is particular. Notice that a $\mathbb{Z}_{2}\times\mathbb{Z}_{2}%
$-coloring of any cubic graph must be proper, since $(1,0),$ $(0,1)$ and
$(1,1)$ is the only combination of three non-zero elements of $\mathbb{Z}%
_{2}\times\mathbb{Z}_{2}$ which sums to $(0,0).$ This implies that
$\mathbb{Z}_{2}\times\mathbb{Z}_{2}$-coloring is equivalent to a proper
$3$-edge-coloring of a cubic graph $G,$ and it is well known that not all
cubic graphs have such a coloring. Hence, it is natural to pose a question
whether a proper abelian colorings of cubic graphs exists and for which
abelian groups. In \cite{Macajova2005}, it was noticed that there exist groups
which do not color all bridgeless cubic graphs, these are cyclic groups of
order smaller than 10, and that all abelian groups of order at least 12 do.
The existence of proper Abelian colorings by groups $\mathbb{Z}_{4}%
\times\mathbb{Z}_{2}$, $\mathbb{Z}_{3}\times\mathbb{Z}_{3}$, $\mathbb{Z}_{10}$
and $\mathbb{Z}_{11}$ remains open, and these four groups are called
\emph{exceptional}. The following conjecture is proposed by
M\'{a}\v{c}ajov\'{a} et al. in \cite{Macajova2005}, and then reconsidered by
Kr\'{a}l' et al. in \cite{Kral2009}.

\begin{conjecture}
\label{Con_exceptional}Every bridgeless cubic graph $G$ has a proper
$A$-coloring for every exceptional abelian group $A$.
\end{conjecture}

Moreover, in \cite{Kral2009} the above conjecture is approached from the
direction of studying relationships between Steiner colorings and abelian
colorings. For every abelian group $A$, there exists a partial Steiner triple
system $\mathcal{C}(A)$ whose points are all non-zero elements of $A$ and
blocks are all $3$-element subsets of $A-\{0\}$ with zero sum, so
$\mathcal{C}(A)$-colorings coincide with proper $A$-colorings. Two Steiner
triple systems are particularly studied, the projective plane with $7$ points
(the Fano plane), and the affine plane with 9 points. It is established that
for $i\in\{4,5,6\}$ each of these systems contains a configuration of $i$
lines, denoted by $F_{i}$ and $A_{i}$, respectively, such that a cubic graph
is $F_{i}$-colorable if and only if it is $A_{i}$-colorable. This result has
several important consequences. First, it is known that every bridgeless cubic
graph has an $F_{6}$-coloring \cite{MacajovaFano}, so this result implies that
it also has an $A_{6}$-coloring. Further, the result of Holroyd and
\v{S}koviera \cite{Holroyd2004}, asserting that every bridgeless cubic graph
has an $S$-coloring for every non-trivial Steiner triple system $S$ is also
implied. Finally, regarding the four exceptional groups it is noted that
$\mathcal{C}(\mathbb{Z}_{4}\times\mathbb{Z}_{2})$ coincides with $F_{5}$.
Although this configuration is not contained in other exceptional groups, the
configuration $A_{5}$ is, so the following theorem is also established in
\cite{Kral2009}.

\begin{theorem}
If a cubic graph has a proper $A$-coloring for $A=\mathbb{Z}_{4}%
\times\mathbb{Z}_{2}$, then it has a proper $A$-coloring for every
$A\in\{\mathbb{Z}_{3}\times\mathbb{Z}_{3},\mathbb{Z}_{10},\mathbb{Z}_{11}\}$.
\end{theorem}

\noindent The above theorem implies that it is sufficient to establish that
every cubic graph has a proper $\mathbb{Z}_{4}\times\mathbb{Z}_{2}$-coloring
in order to establish Conjecture \ref{Con_exceptional}.

In this paper we consider exclusively proper $\mathbb{Z}_{4}\times
\mathbb{Z}_{2}$-colorings of cubic graphs, hence if we say only $\mathbb{Z}%
_{4}\times\mathbb{Z}_{2}$-coloring we tacitly assume proper $\mathbb{Z}%
_{4}\times\mathbb{Z}_{2}$-coloring. First, for a given matching $M$ in a
$2$-factor $F$ of a snark $G$ we define the notions of an $F$-matching and the
corresponding $F$-complement. Then we give a characterization of proper
$\mathbb{Z}_{4}\times\mathbb{Z}_{2}$-colorings in terms of the existence of an
$F$-matching such that $F$-complement has a particular property. Using this
characterization, we show for example that some oddness two snarks have a
proper $\mathbb{Z}_{4}\times\mathbb{Z}_{2}$-coloring. For general snarks, we
further provide a sufficient condition for the existence of an $F$-matching
and a sufficient condition under which it can be modified to have the property
required by the introduced characterization. Since the characterization of
proper $\mathbb{Z}_{4}\times\mathbb{Z}_{2}$-colorings depends on the chosen
$2$-factor $F$ and the chosen matching $M$ in it, given that for a same snark
$G$ there exist multiple choices for $F$ and $M,$ the snarks not covered by
the results of this paper should be rare. Finally, the introduced sufficient
conditions applied to oddness two snarks yield that every such snark which has
a $2$-factor with precisely two odd cycles and at most one even cycle has a
proper $\mathbb{Z}_{4}\times\mathbb{Z}_{2}$-coloring. Extending the two
sufficient conditions to the proof that all snarks indeed have an $F$-matching
whose $F$-complement has the desired properties we leave for further work.

\section{A characterization of proper $\mathbb{Z}_{4}\times\mathbb{Z}_{2}%
$-colorings}

In this section we give a characterization of $\mathbb{Z}_{4}\times
\mathbb{Z}_{2}$-colorable graphs via matchings in a $2$-factor of a cubic
graph. More precisely, we will establish that a cubic graph $G$ is
$\mathbb{Z}_{4}\times\mathbb{Z}_{2}$-colorable if by consecutive edge
reductions along a specific matching $M$ of $G$ it reduces to a $3$-colorable
graph. Not any matching $M$ of $G$ will suffice, $M$ has to be a matching in a
$2$-factor $F$ of $G$ such that $H=G-M$ has a particular form. Before we
proceed, let us say that the property of a proper abelian coloring that the
sum of the three colors at each vertex equals zero is also called the
\emph{zero sum} property, and such coloring is called a \emph{zero sum}
coloring. We first need the following auxiliary lemma. For a $\mathbb{Z}%
_{k}\times\mathbb{Z}_{2}$-coloring $\sigma=(x,y)$ of a snark $G,$ let
$X_{i}=\{e\in E(G):x(e)=i\}$ for $i\in\mathbb{Z}_{k},$ and $Y_{j}=\{e\in
E(G):y(e)=j\}$ for $j\in\mathbb{Z}_{2}.$

\begin{lemma}
\label{Lemma_PM}Let $\sigma=(x,y)$ be a $\mathbb{Z}_{k}\times\mathbb{Z}_{2}%
$-coloring of a snark $G.$ The set $Y_{0}=\{e\in E(G):y(e)=0\}$ is a perfect
matching in $G$ for every $k\in\{4,5\}.$
\end{lemma}

\begin{proof}
Let $v$ be any vertex of $G$ and $e_{1},$ $e_{2}$ and $e_{3}$ the three edges
incident to $v.$ By the definition of a proper abelian coloring it must hold
$y(e_{1})+y(e_{2})+y(e_{3})=0,$ which implies that either all three colors
$y(e_{1}),$ $y(e_{2}),$ $y(e_{3})$ are equal to $0,$ or precisely one of them
is equal to $0.$

Assume first that $y(e_{1})=y(e_{2})=y(e_{3})=0.$ This implies that
$x(e_{i})\not =0$ for every $i\in\{1,2,3\},$ otherwise $\sigma$ contains an
edge colored by $(0,0),$ a contradiction. If $x(e_{i})=x(e_{j})$ for some
$i\not =j,$ then $\sigma(e_{i})=\sigma(e_{j})$, so $\sigma$ would not be
proper. Hence, $x$ must also be proper at $v,$ i.e. $x(e_{1}),$ $x(e_{2})$ and
$x(e_{3})$ must be three pairwise distinct non-zero elements of $\mathbb{Z}%
_{k}$ whose sum is zero, and such three elements do not exist in
$\mathbb{Z}_{k}$ for any $k\in\{4,5\}.$

The only remaining possibility is that precisely one of the colors $y(e_{1}),$
$y(e_{2}),$ $y(e_{3})$ is equal to $0,$ so $Y_{0}=\{e\in E(G):y(e)=0\}$ is a
perfect matching in $G.$
\end{proof}

Before we proceed with the main result of this section, we first need several
notions. First, an \emph{end-edge} of a path is an edge of the path incident
to an end-vertex. Notice that any path of length at least two has precisely
two end-edges. Let $G$ be a snark, $F$ a $2$-factor in $G,$ $M$ a matching in
$F,$ and $H=G-M.$ Notice that $H$ contains vertices of degree $2$ and $3,$
which will shortly be called $2$\emph{-vertices} and $3$\emph{-vertices},
respectively. An $F$\emph{-path} in $H$ is a path $P$ in $H$ for which the
following hold:

\begin{enumerate}
\item end-vertices of $P$ are $3$-vertices of $H,$

\item all interior vertices of $P$ are $2$-vertices of $H,$

\item end-edges of $P$ belong to $F.$
\end{enumerate}

\noindent An $F$\emph{-matching} of $H$ is any collection of vertex disjoint
$F$-paths in $H$ such that every $3$-vertex of $H$ is incident to precisely
one of the paths in the collection. For a given $F$-matching of $H,$ an
$F$\emph{-complement} is a subgraph of $H$ induced by all edges of $H$ not
contained in the paths of the $F$-matching. Notice that an $F$-complement is a
collection of vertex disjoint cycles in $H$ and all $3$-vertices of $H$ are
contained on cycles in the $F$-complement. In order to distinguish them from
cycles in $F,$ the cycles of the $F$-complement in $H$ will be called
\emph{loops}. A loop of an $F$-complement is $3$\emph{-even} if it contains an
even number of $3$-vertices of $H,$ otherwise it is $3$\emph{-odd}. An
$F$-complement is $3$\emph{-even} if all its loops are $3$-even. We can now
characterize proper $\mathbb{Z}_{4}\times\mathbb{Z}_{2}$-colorings of cubic
graphs as follows.

\begin{theorem}
\label{Lemma_karakterizacija}A cubic graph $G$ has a proper $\mathbb{Z}%
_{4}\times\mathbb{Z}_{2}$-coloring if and only if there exists a $2$-factor
$F$ in $G$ and a matching $M$ in $F$ such that $H=G-M$ has an $F$-matching
whose $F$-complement is $3$-even.
\end{theorem}

\begin{proof}
Let $G$ be a $\mathbb{Z}_{4}\times\mathbb{Z}_{2}$-colorable cubic graph. We
will provide $F$ and $M$ with the desired properties. Denote by $\sigma=(x,y)$
a $\mathbb{Z}_{4}\times\mathbb{Z}_{2}$-coloring of $G.$ Lemma \ref{Lemma_PM}
implies that $Y_{0}=\{e\in E(G):y(e)=0\}$ is a perfect matching in $G,$ so
$F=G-Y_{0}$ is a $2$-factor in $G.$ Notice that $X_{0}\cap Y_{0}=\emptyset,$
otherwise $\sigma$ would contain an edge colored by $(0,0)$. Hence,
$X_{0}\subseteq E(F).$ Let us further show that\textit{ }$X_{0}$\emph{ is a
matching in }$F.$ Assume to the contrary that there are two edges $e_{1}$ and
$e_{2}$ in $X_{0}$ incident to a same vertex $v$ of $F.$ Since $X_{0}\cap
Y_{0}=\emptyset$ and $e_{1},e_{2}\in X_{0},$ it follows that $e_{1}%
,e_{2}\not \in Y_{0},$ hence $y(e_{1})=y(e_{2})=1.$ Also, since $e_{1}$ and
$e_{2}$ belong to $X_{0},$ we have $x(e_{1})=x(e_{2})=0.$ We conclude that
$\sigma(e_{1})=\sigma(e_{2})=(0,1),$ so $\sigma$ is not proper, a
contradiction. Hence, $X_{0}$ is a matching in $F,$ so we denote $M=X_{0}.$

Recall that $X_{2}=\{e\in E(G):x(e)=2\}$ and let us establish that $X_{2}%
$\emph{ is an }$F$\emph{-matching of }$H=G-M.$ Since $x$ is a zero sum
coloring of $G,$ every $2$-vertex of $H$ is incident either to two or to none
of the edges of $X_{2}.$ Also, colors of $x$ on the three edges incident to
any $3$-vertex of $H$ are either $2,1,1$ or $2,3,3,$ so every $3$-vertex of
$H$ is incident to precisely one edge of $X_{2}.$ Hence, $X_{2}$ is a
collection of vertex disjoint paths in $H$ whose end-vertices are $3$-vertices
of $H$ and cover all of them.

It remains to show that \emph{every path of }$X_{2}$\emph{ begins and ends
with an edge of }$F,$ i.e. that every edge of $X_{2}$ incident to a $3$-vertex
of $H$ is included in $F.$ Assume to the contrary, that there exists an
end-vertex $v$ of a path of $X_{2},$ such that the only edge $e_{1}$ of
$X_{2}$ incident to $v$ is not included in $F.$ Hence, $e_{1}\in Y_{0}$ and
$y(e_{1})=0.$ Let $e_{2}$ and $e_{3}$ be the remaining two edges of $G$
incident to $v,$ then $y(e_{2})=y(e_{3})=1.$ Since $e_{1}$ belongs to $X_{2},$
it follows that $x(e_{1})=2.$ Also, since $v$ is a $3$-vertex of $H,$ this
implies $v$ is not incident to edges of $M,$ hence either $x(e_{2}%
)=x(e_{3})=1$ or $x(e_{2})=x(e_{3})=3.$ Either way, we obtain $\sigma
(e_{2})=\sigma(e_{3}),$ a contradiction with $\sigma$ being proper. We
conclude that $X_{2}$ is a $F$-matching of $H$ and we denote $X_{2}%
=\mathcal{P}.$

To establish that $\mathbb{Z}_{4}\times\mathbb{Z}_{2}$-colorable cubic graphs
have a $2$-factor $F$ and a matching $M$ in $F$ with desired properties, it
remains to establish that \emph{the }$F$\emph{-complement }$\mathcal{L}$\emph{
of }$\mathcal{P}$\emph{ is }$3$\emph{-even}. Assume to the contrary that it is
not $3$-even. Let $L$ be a loop of $\mathcal{L}$ which contains odd number of
$3$-vertices of $H.$ Let $v$ be a vertex of $H$ included in $L$ and $e_{1}$
and $e_{2}$ the two edges of $L$ incident to $v.$ If $v$ is a $2$-vertex of
$H$, then it must hold that $x(e_{1})=1$ and $x(e_{2})=3$ or vice versa, since
the third edge $e_{3}$ of $G$ incident to $v$ is included in $M,$ hence
$x(e_{3})=0.$ Otherwise, if $v$ is a $3$-vertex of $H$ included in $L,$ then
it must be either $x(e_{1})=x(e_{2})=1$ or $x(e_{1})=x(e_{2})=3,$ since the
third edge $e_{3}$ of $G$ incident to $v$ is an end-edge of a path of
$\mathcal{P}$ and hence $x(e_{3})=2.$ By the zero sum property the sum of the
$x$ coordinate of the edges at each vertex is $0$, hence the number of edges
that are going out of $L$ with value $x(e)=2$ must be even, and consequently
the number of $3$-vertices of $F$ on $L$ is even. \bigskip

Let us now prove the other direction of the equivalence, i.e. let us assume
that there exists a $2$-factor $F$ in $G$ and a matching $M$ in $F$ such that
$H=G-M$ has an $F$-matching $\mathcal{P}$ whose $F$-complement $\mathcal{L}$
is $3$-even. We define a $\mathbb{Z}_{4}\times\mathbb{Z}_{2}$-coloring
$\sigma=(x,y)$ of $G$ as follows. First we define $y(e)=1$ if $e\in F,$ and
$y(e)=0$ otherwise. Obviously, $y$ is a zero sum coloring of $G.$ Next, we
define $x(e)=0$ if and only if $e\in M,$ also $x(e)=2$ if and only if $e\in
E(\mathcal{P}).$ It remains to define the values of $x$ on the edges of
$\mathcal{L}.$ Notice that $\mathcal{L}$ is a collection of loops of $H,$ i.e.
cycles of $H$ which are also cycles in $G.$ Since $\mathcal{P}$ is an
$F$-matching, this implies that for every $3$-vertex of $H$ which belongs to a
loop of $\mathcal{L}$ precisely one edge incident to it in $\mathcal{L}$
belongs to $F.$ Also, for every $2$-vertex of $H$ which belongs to
$\mathcal{L}$ precisely one of the two edges incident to it in $\mathcal{L}$
belong to $\mathcal{L}$. This implies that every loop in $\mathcal{L}$ is of
an even length. If a loop $L$ of $\mathcal{L}$ contains zero $3$-vertices of
$H,$ then we define the value of $x$ on the edges of such loop to be $1$ and
$3$ alternatively, and this can be done since $L$ is of an even length. On the
other hand, if a loop $L$ of $\mathcal{L}$ contains a $3$-vertex of $H,$ we
choose a $3$-vertex $v_{0}$ which belongs to $L$ and we denote $L=v_{0}%
v_{1}\cdots v_{k}v_{0}.$ We first define $x(v_{0}v_{1})=1,$ and then we define
$x(v_{i}v_{i+1})=x(v_{i-1}v_{i})$ if $v_{i}$ is a $3$-vertex of $H$ and
$x(v_{i}v_{i+1})\in\{1,3\}\backslash\{x(v_{i-1}v_{i})\}$ if $v_{i}$ is a
$2$-vertex of $H.$ Notice that thus defined $x$ takes only values from
$\{1,3\}$ on the edges of $L.$ Also, notice that the two edges of $L$ incident
to a $3$-vertex of $H$ always have the same value of $x,$ namely for a
$3$-vertex $v_{i}$ of $L$ where $i\not =0$ this holds by the definition of
$x,$ and for $v_{0}$ of $L$ this holds since $F$-complement $\mathcal{L}$ is
$3$-even so $L$ contains an even number of $3$-vertices of $H.$

We have now defined a $\mathbb{Z}_{4}\times\mathbb{Z}_{2}$-coloring
$\sigma=(x,y)$ of $G,$ and we further have to establish that $\sigma$ is a
proper nowhere-zero zero sum coloring of $G.$ Let us first establish that
$\sigma$ is a zero sum coloring of $G,$ and to do so it is sufficient to show
that $x$ is a zero sum coloring, since we already established that $y$ is a
zero sum coloring. Let $v$ be a vertex of $H$ and $e_{1},$ $e_{2}$ and $e_{3}$
its three incident edges. We distinguish the following three cases.

\smallskip\noindent\textbf{Case 1:} $v$\emph{ is a }$3$\emph{-vertex of }$H.$
Notice that precisely one of $e_{1},$ $e_{2}$ and $e_{3}$ belongs to
$\mathcal{P},$ say $e_{1},$ so $x(e_{1})=2.$ The remaining two edges incident
to $v$ belong to a same loop $L$ of $\mathcal{L},$ so $x(e_{2})=x(e_{3})=1$ or
$x(e_{2})=x(e_{3})=3.$ Either way, we have $x(e_{1})=x(e_{2})=x(e_{3})=0.$

\smallskip\noindent\textbf{Case 2:} $v$\emph{ is a }$2$\emph{-vertex of }%
$H$\emph{ which belongs to }$\mathcal{P}$\emph{.} In this case precisely one
of $e_{1},$ $e_{2}$ and $e_{3}$ belongs to $M,$ say $e_{1},$ so $x(e_{1})=0.$
The other two edges incident to $v$ belong to $\mathcal{P},$ so $x(e_{2}%
)=x(e_{3})=2.$ We again have $x(e_{1})=x(e_{2})=x(e_{3})=0.$

\smallskip\noindent\textbf{Case 3:} $v$\emph{ is a }$2$\emph{-vertex of }%
$H$\emph{ which belongs to }$\mathcal{L}.$ We may again assume that $e_{1}\in
M,$ so $x(e_{1})=0.$ Edges $e_{1}$ and $e_{2}$ incident to $v$ belong to a
same loop $L$ of $\mathcal{L},$ so $x(e_{2})=1$ and $x(e_{3})=3,$ which again
implies $x(e_{1})=x(e_{2})=x(e_{3})=0.$

\medskip

\noindent Thus, we have established that $\sigma$ is a zero sum coloring of
$G.$

\medskip

Let us next establish that $\sigma$\emph{ is nowhere-zero and proper}. Notice
that $x(e)=0$ if and only if $e\in M$, and $y(e)=0$ if and only if
$e\not \in F.$ Since $M\subseteq E(F),$ we conclude that $\sigma$ is
nowhere-zero. It remains to establish that $\sigma$ is proper at each vertex
of $G.$ Notice that $V(G)=V(F)$, so it is sufficient to consider vertices of
$F.$ Let $v\in V(F)$ and let $e_{1},$ $e_{2}$ and $e_{3}$ be the three edges
of $G$ incident to $F$ denoted so that $e_{1},e_{2}\in F$ and $e_{3}%
\not \in F.$ This implies that $y(e_{1})=y(e_{2})=1$ and $y(e_{3})=0$.
In\ order to establish that $\sigma$ is proper it is sufficient to show that
$x(e_{1})\not =x(e_{2}).$ We distinguish the following two cases.

\smallskip\noindent\textbf{Case 1: }$v$\emph{ is incident to an edge of
}$M\subseteq E(F).$ Since $e_{1},e_{2}\in F$ are the two edges of $F$ incident
to $v,$ we conclude that precisely one of the two belongs to $M,$ say $e_{1},$
so $x(e_{1})=0.$ Since $M$ is a matching in $F,$ it follows that
$e_{2}\not \in M,$ so $x(e_{2})\not =0,$ hence $x(e_{1})\not =x(e_{2}).$

\smallskip\noindent\textbf{Case 2: }$v$\emph{ is not incident to an edge of
}$M\subseteq E(F).$ In this case $v$ is a $3$-vertex of $H=G-M.$ This implies
that one edge incident to $v$ belongs to $\mathcal{P}$ and the remaining two
edges incident to $v$ belong to $\mathcal{L}.$ Since $\mathcal{P}$ is an
$F$-matching, we know that $e_{3}$ does not belong to $\mathcal{P},$ as
$e_{3}$ does not belong to $F.$ We may assume $e_{1}$ belongs to $\mathcal{P}$
and $e_{2}$ belongs to $\mathcal{L}.$ Then $x(e_{1})=2$ and $\sigma_{1}%
(e_{2})\in\{1,3\},$ so $x(e_{1})\not =x(e_{2}).$

\smallskip We have established that $\sigma$ is also proper, so we conclude
that $\sigma$ is a proper $\mathbb{Z}_{4}\times\mathbb{Z}_{2}$-coloring of
$G.$
\end{proof}

Notice that a $3$-edge-coloring of a cubic graphs immediately yields a
$\mathbb{Z}_{4}\times\mathbb{Z}_{2}$-coloring if the three colors are replaced
by $(1,0),$ $(1,1)$ and $(2,1).$ Hence, we will apply Theorem
\ref{Lemma_karakterizacija} mainly to snarks. Nevertheless, in order to
illustrate the use of the above theorem on a simple example, we provide the
following corollary.

\begin{corollary}
Every $3$-edge-colorable cubic graph $G$ has a proper $\mathbb{Z}_{4}%
\times\mathbb{Z}_{2}$-coloring.
\end{corollary}

\begin{proof}
Let $F$ be a $2$-factor in $G$ induced by all edges colored by $1$ and $2.$
Since every cycle in $F$ is of an even length, there exists a perfect matching
$M$ in $F.$ Each edge of $E(F)\backslash M$ is an $F$-path in $H$, so
$\mathcal{P=}E(F)\backslash M$ is an $F$-matching of $H.$ Notice that the
$F$-complement $\mathcal{L}$ of $\mathcal{P}$ consists of even cycles with
zero $3$-vertices of $H,$ hence Theorem \ref{Lemma_karakterizacija} implies
that $G$ is $\mathbb{Z}_{4}\times\mathbb{Z}_{2}$-colorable.
\end{proof}

\begin{figure}[h]
\begin{center}
\raisebox{-0.9\height}{\includegraphics[scale=0.6]{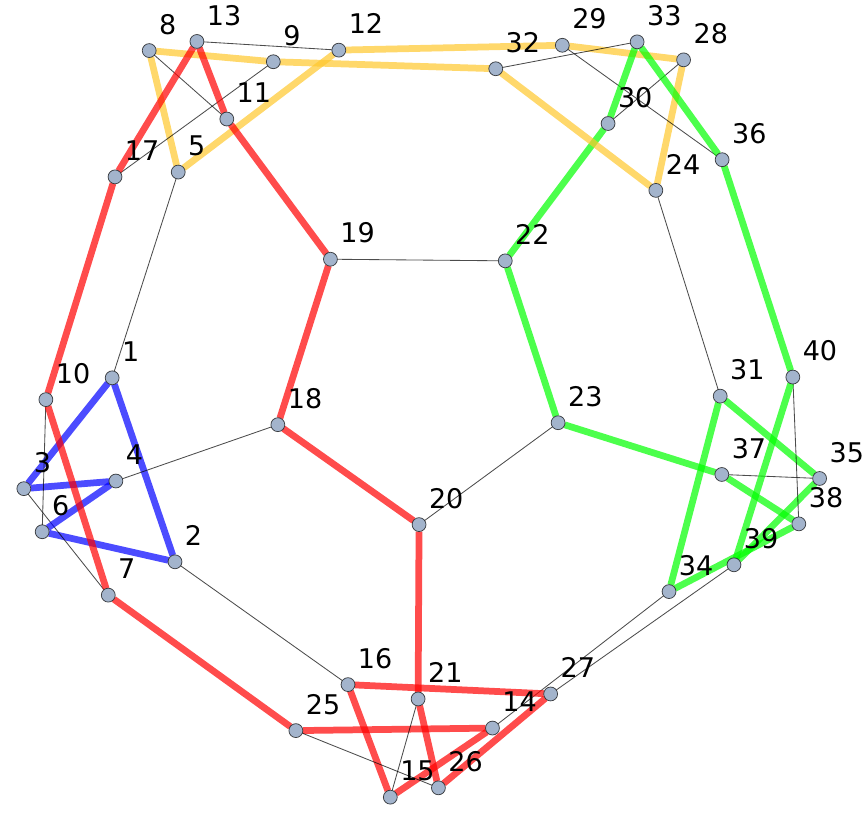}}
\end{center}
\caption{An oddness two snark $G$ on 40 vertices and a $2$-factor $F$ in $G$
with two odd and two even cycles, where different cycles of $F$ are shown in
different colors.\ }%
\label{Fig01}%
\end{figure}

\begin{figure}[ph]
\begin{center}
\raisebox{-0.9\height}{\includegraphics[scale=0.6]{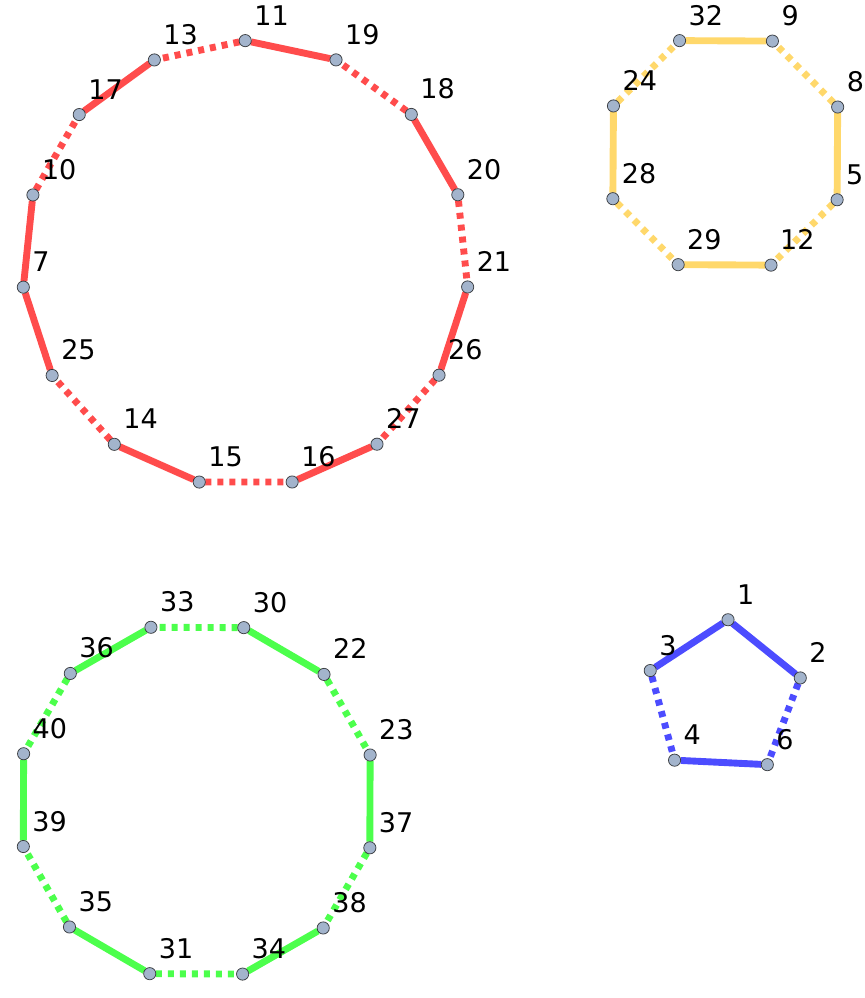}}
\end{center}
\caption{For the snark $G$ and the $2$-factor $F$ in $G$ depicted in Figure
\ref{Fig01}, this figure shows a maximum matching $M$ in $F$ which consists of
dashed edges.}%
\label{Fig02}%
\end{figure}

\begin{figure}[ph]
\begin{center}
\raisebox{-0.9\height}{\includegraphics[scale=0.6]{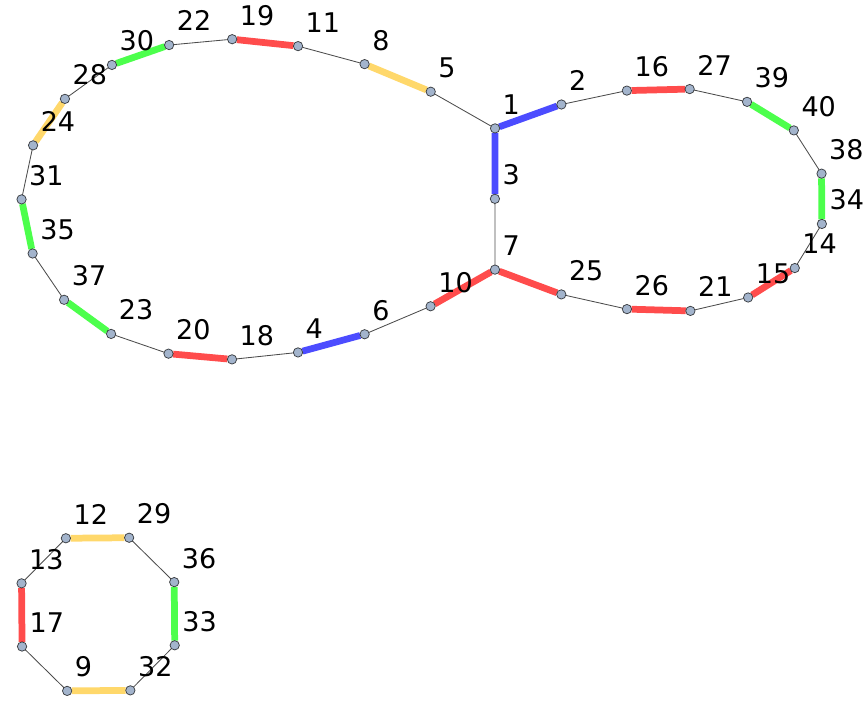}}
\end{center}
\caption{For the snark $G$ from Figure \ref{Fig01}, and the maximum matching
$M$ in $F$ from Figure \ref{Fig02}, this figure shows the corresponding graph
$H=G-M$. Notice that the main component of $H$ is a $\Theta$-graph.}%
\label{Fig03}%
\end{figure}

\section{Applying the characterization and two obstacles}

Given a snark $G$ and a $2$-factor $F$ in $G,$ for some choices of $M$ the
graph $H$ may satisfy the conditions of Theorem \ref{Lemma_karakterizacija},
so the snark $G$ has a proper $\mathbb{Z}_{4}\times\mathbb{Z}_{2}$-coloring.
For other choices of $M$ the following two problems with the application of
Theorem \ref{Lemma_karakterizacija} may arise, namely:

\begin{itemize}
\item[P1)] the graph $H=G-M$ may not have an $F$-matching;

\item[P2)] if the graph $H$ has an $F$-matching, the $F$-complement may not be
$3$-even.
\end{itemize}

\noindent In what follows, we give an illustration of all these possibilities
and to keep thing simple we restrict our consideration to oddness two snarks.

\paragraph{Oddness two snarks and the choice of $M$ which avoids problems P1)
and P2).}

Let $G$ be an oddness two snark, $F$ a $2$-factor in $G$ with precisely two
odd cycles, and $M$ a maximum matching in $F.$ Notice that $M$ covers all
vertices of $F$ except one vertex from each odd cycle of $F.$ Hence, $H=G-M$
contains precisely two vertices of degree three and all other vertices of $H$
are of degree $2.$ The graph $H$ does not have to be connected, but the pair
of $3$-vertices of $H$ must belong to a same connected component of $H$ which
is called the \emph{main} component of $H.$ All other connected components of
$H$ are called \emph{secondary}.

Let $v$ be any $2$-vertex of $H$ and $e_{1},$ $e_{2}$ the two edges of $H$
incident to $v.$ Since the third edge $e_{3}$ of $G$ incident to $v$ belongs
to $M\subseteq E(F),$ we conclude that precisely one of the edges $e_{1}$ and
$e_{2}$ belongs to $F.$ This implies that every secondary component of $H$ is
an even length cycle. Also, the main component of $H,$ since it contains
precisely two vertices of degree $3$ and all other vertices are of degree $2,$
is either a $\Theta$-graph or a \emph{kayak-paddle} graph, i.e. a graph
consisting of two vertex disjoint cycles connected by a path.

\begin{figure}[h]
\begin{center}
\raisebox{-0.9\height}{\includegraphics[scale=0.6]{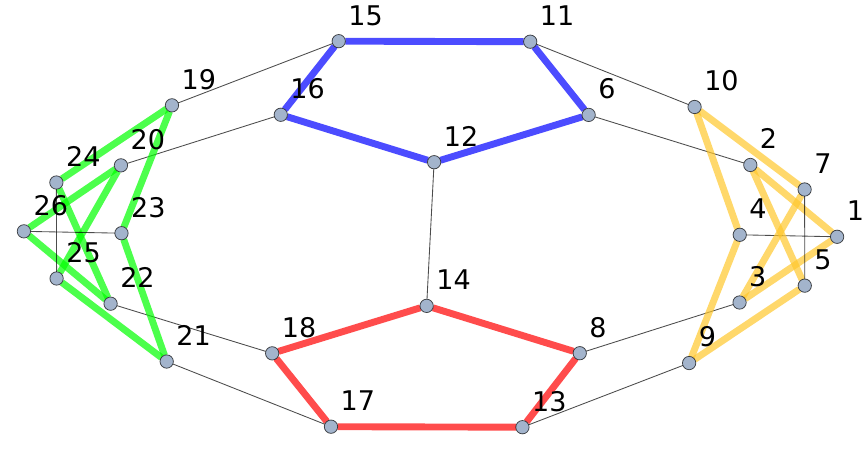}}
\end{center}
\caption{An oddness two snark $G$ on 26 vertices and a $2$-factor $F$ in $G$
with four cycles, two of which are odd. The different cycles of $F$ are shown
in different colors.}%
\label{Fig04}%
\end{figure}

\begin{figure}[ph]
\begin{center}
\raisebox{-0.9\height}{\includegraphics[scale=0.6]{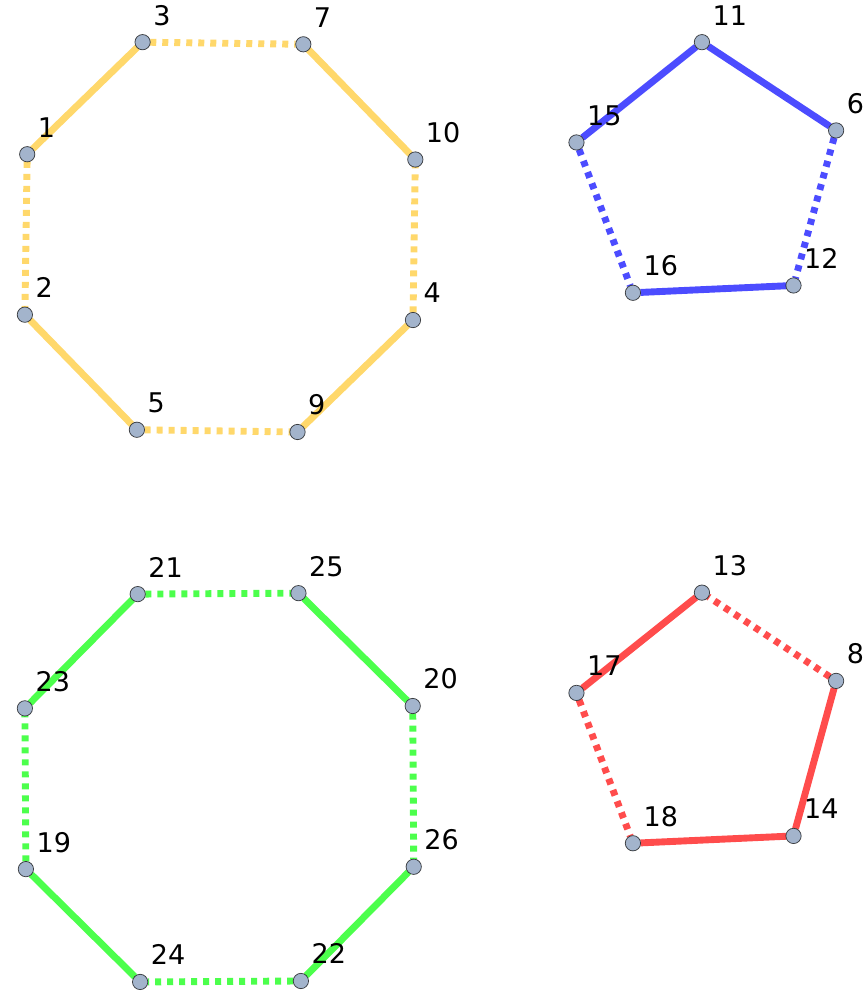}}
\end{center}
\caption{For the snark $G$ and $2$-factor $F$ in $G$ depicted in Figure
\ref{Fig04}, this figure shows a maximum matching $M$ in $F$ which consists of
dashed edges such that the graph $H=G-M$ shown in Figure \ref{Fig06a} does not
have an $F$-matching.}%
\label{Fig05a}%
\end{figure}

\begin{figure}[ph]
\begin{center}
\raisebox{-0.9\height}{\includegraphics[scale=0.6]{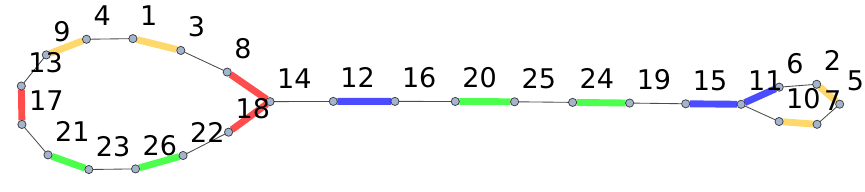}}
\end{center}
\caption{For the snark $G$ and $2$-factor $F$ in $G$ as in Figure \ref{Fig04},
and the maximal matching $M$ in $F$ as in Figure \ref{Fig05a}, this figure
shows the corresponding graph $H=G-M$. Notice that the main component of $H$
is a kayak paddle graph such that the path $P$ connecting the pair of
$3$-vertices in $H$ is not an $F$-path, so $H$ does not have an $F$-matching.}%
\label{Fig06a}%
\end{figure}

Let us first consider the case when the main component of $H$ is a $\Theta
$-graph. This case is illustrated by the snark $G$ from Figure \ref{Fig01}
with the cycles of the $2$-factor $F$ in $G$ highlighted in different colors.
A maximum matching $M$ in $F$ consists of the dashed edges from Figure
\ref{Fig02}, and the resulting graph $H=G-M$ is shown in Figure \ref{Fig03}.
Obviously, $H$ has two connected components, the main component is a $\Theta
$-graph and the only secondary component is an even length cycle. In the
following proposition we show that in the case when the main component is a
$\Theta$-graph, at least one path connecting the pair of $3$-vertices is an
$F$-path, hence an $F$-matching, and the corresponding $F$-complement is
$3$-even, i.e. such a matching $M$ has neither of the problems P1) and P2), so
$G$ has a proper $\mathbb{Z}_{4}\times\mathbb{Z}_{2}$-coloring.

\begin{proposition}
\label{Lema_theta}Let $G$ be an oddness two snark, $F$ a $2$-factor in $G$
with precisely two odd cycles, and $M$ a maximum matching in $F.$ If the main
component of $H=G-M$ is a $\Theta$-graph, then $G$ has a proper $\mathbb{Z}%
_{4}\times\mathbb{Z}_{2}$-coloring.
\end{proposition}

\begin{proof}
Let $r_{1}$ and $r_{2}$ be the two vertices of degree $3$ in $H$ which belong
to the main component of $H.$ For each of the vertices $r_{1}$ and $r_{2}$ it
holds that two out of the three edges incident to it in $H$ belong to $F.$
Hence, at least one of the three paths connecting $r_{1}$ and $r_{2}$ in the
main component of $H$ is an $F$-path and forms an $F$-matching $\mathcal{P}$.
Also, the $F$-complement $\mathcal{L}$ of $\mathcal{P}$ is $3$-even, since
every secondary component of $H$ contains zero $3$-vertices of $H,$ and the
cycle of $\mathcal{\mathcal{L}}$ which is a subgraph of the main component of
$H$ contains precisely two $3$-vertices of $H.$ Hence, Theorem
\ref{Lemma_karakterizacija} implies the claim.
\end{proof}

\paragraph{Oddness two snarks and the choice of $M$ with the problem P1).}

Let us again consider an oddness two snark $G$ and a $2$-factor $F$ in $G$
with precisely two odd cycles. This time we consider a maximum matching $M$ of
$F$ such that the main component of $H=G-M$ is a kayak-paddle graph. We will
provide an example of a maximum matching $M$ in $F,$ such that the only path
$P$ in $H$ which connects a pair of $3$-vertices is not an $F$-path, i.e. it
either does not begin or it does not end with an edge of $F.$ Hence, such a
graph $H$ does not have an $F$-matching.

\begin{figure}[ph]
\begin{center}
\raisebox{-0.9\height}{\includegraphics[scale=0.6]{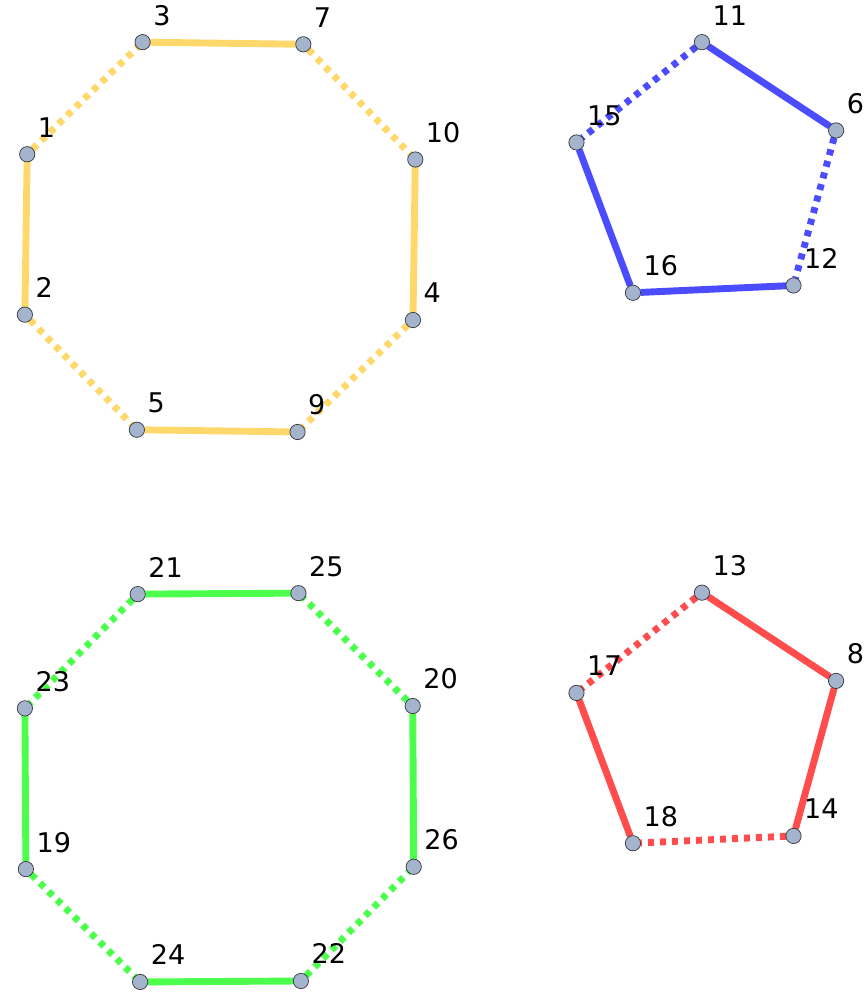}}
\end{center}
\caption{For the same snark $G$ from Figure \ref{Fig04} and the $2$-factor $F$
highlighted in it, this figure shows a maximum matching $M$ in $F$ distinct
from the one in Figure \ref{Fig06a} for which the corresponding graph $H=G-M$
is shown in Figure \ref{Fig08}.}%
\label{Fig07a}%
\end{figure}

\begin{figure}[ph]
\begin{center}
\raisebox{-0.9\height}{\includegraphics[scale=0.6]{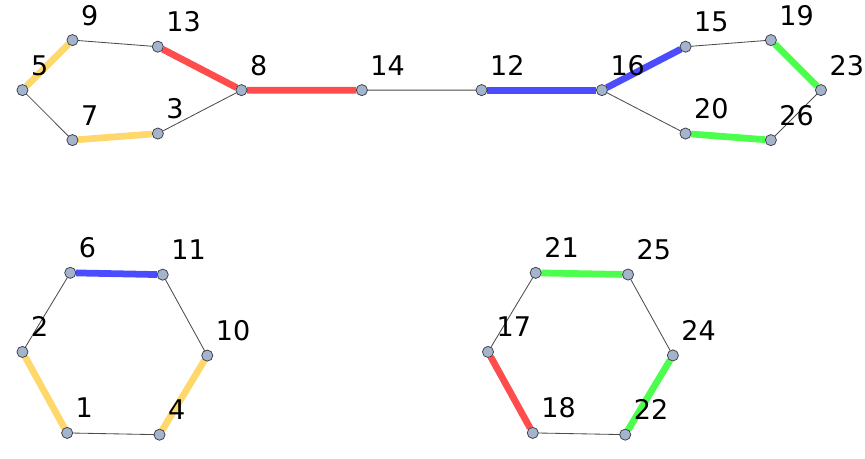}}
\end{center}
\caption{For the snark $G$ from Figure \ref{Fig04}, and the maximum matching
$M$ in $F$ from Figure \ref{Fig07a}, this figure shows the corresponding graph
$H=G-M$. Notice that the main component of $H$ is a kayak-paddle graph such
that the path $P$ connecting the pair of $3$-vertices of $H$ is an $F$-path.
Hence, $\mathcal{P}=\{P\}$ is an $F$-matching of $H$ and the $F$-complement
$\mathcal{L}$ of $\mathcal{P}$ is not $3$-even since it contains a pair of
$3$-odd loops.}%
\label{Fig08}%
\end{figure}

Take as an example the snark $G$ from Figure \ref{Fig04} where the maximum
matching $M$ in $F$ is chosen as in Figure \ref{Fig05a}. The corresponding
graph $H=G-M$ is shown in Figure \ref{Fig06a}. Notice that the $r_{1}=14$ and
$r_{2}=11$ is the pair of $3$-vertices of $H,$ the edge incident to $r_{2}=11$
on the path $P$ connecting $r_{1}$ and $r_{2}$ does belong to $F,$ but the
edge of $P$ incident to $r_{1}=14,$ does not belong to $F.$ Hence, $H$ does
not have an $F$-matching, so this choice of $M$ has the problem P1). In the
next section we provide a sufficient condition for the existence of a maximum
matching $M$ in a $2$-factor of a general snark $G$ such that $H=G-M$ has an
$F$-matching.

\paragraph{Oddness two snarks and the choice of $M$ with the problem P2).}

Let us again consider an oddness two snark $G$ and a $2$-factor $F$ in $G$
with precisely two odd cycles. Here, we show that a maximum matching $M$ in
$F$ can be chosen so that the main component of $H=G-M$ is a kayak-paddle
graph which has an $F$-matching, but the corresponding $F$-complement is not
$3$-even, i.e. which has the problem P2).

An example of such a snark $G$ is again the snark from Figure \ref{Fig04}, but
this time we choose a maximum matching $M$ in $F$ as in Figure \ref{Fig07a}.
The corresponding graph $H=G-M$ is shown in Figure \ref{Fig08} and it is
obviously a kayak-paddle graph. Notice that the only path $P$ of $H$
connecting the pair of $3$-vertices $r_{1}=8$ and $r_{2}=16$ is an $F$-path,
hence $\mathcal{P}=\{P\}$ is an $F$-matching, but its $F$-complement
$\mathcal{L}$ is not $3$-even. Namely, the loop $L_{1}^{o}$ of $\mathcal{L}$
induced by vertices $\{3,5,7,8,9,13\}$ contains precisely one $3$-vertex of
$H$ and that is $r_{1}=8,$ and the loop $L_{2}^{o}$ of $\mathcal{L}$ induced
by vertices $\{15,16,19,20,23,26\}$ also contains precisely one $3$-vertex of
$H$ and that is $r_{2}=16.$

In order to apply Theorem \ref{Lemma_karakterizacija} in such cases, we wish
to introduce additional $3$-vertices in $H$ which will change the $3$-parity
of the $F$-complement $\mathcal{L}$. This is done by a modification of the
matching $M,$ and we discuss this modification in the next section, not only
for oddness two snarks, but for all snarks.

\section{A sufficient condition for the existence of a proper $\mathbb{Z}%
_{4}\times\mathbb{Z}_{2}$-coloring}

In this section we first provide a sufficient condition for the existence of a
maximum matching $M$ in a $2$-factor $F$ of $G$ such that $H=G-M$ has an
$F$-matching, i.e. a maximum matching which does not have the problem P1).
Moreover, an $F$-matching is \emph{simple} if for every $F$-path it holds that
only edges incident to its end-vertices belong to odd length cycles of $F.$
The sufficient condition we introduce assures the existence of a simple $F$-matching.

Let $F$ be a $2$-factor in $G$ and let $F_{\mathrm{odd}}$ (resp.
$F_{\mathrm{even}}$) be a subgraph of $F$ induced by all odd (resp. even)
length cycles of $F.$ Let $M_{\mathrm{even}}$ be a perfect matching in
$F_{\mathrm{even}}$ and let $H_{\mathrm{odd}}=G-M_{\mathrm{even}}.$ Notice
that $F_{\mathrm{odd}}$ is a subgraph of $H_{\mathrm{odd}},$ all vertices of
$F_{\mathrm{odd}}$ are $3$-vertices in $H_{\mathrm{odd}}$ and all other
vertices of $H_{\mathrm{odd}}$ are $2$-vertices. The graph $K_{\mathrm{odd}}$
obtained from $H_{\mathrm{odd}}$ by contracting all edges of $F_{\mathrm{odd}%
}$ and suppressing all $2$-vertices is called an \emph{odd-cycle-incidence} graph.

\begin{proposition}
\label{Tm_proplemP1}Let $G$ be a snark and $F$ a $2$-factor of $G.$ If there
exists a perfect matching $M_{\mathrm{even}}$ in $F_{\mathrm{even}}$ such that
the corresponding $K_{\mathrm{odd}}$ has a perfect matching, then there exists
a maximum matching $M$ in $F$ such that $H=G-M$ has a simple $F$-matching.
\end{proposition}

\begin{proof}
Notice that an odd-cycle-incidence graph $K_{\mathrm{odd}}$ has even number of
vertices, since $F$ has an even number of odd length cycles. Denote by
$C_{1},\ldots,C_{2k}$ all odd length cycles of $F,$ we may assume the same
labels for the corresponding vertices of $K_{\mathrm{odd}}.$ Let
$M_{K_{\mathrm{odd}}}=\{e_{1},\ldots,e_{k}\}$ be a perfect matching in
$K_{\mathrm{odd}},$ and without loss of generality we may assume that
$e_{i}=C_{2i-1}C_{2i}.$ Notice that every edge $e_{i}$ of $M_{K}$ corresponds
to a path $P_{i}$ in $H_{\mathrm{odd}}$ connecting a pair of vertices
$c_{2i-1}\in C_{2i-1}$ and $c_{2i}\in C_{2i}$ such that all interior vertices
of $P_{i}$ are $2$-vertices in $H_{\mathrm{odd}}.$

Since $M_{K}$ is a perfect matching, this implies that every odd cycle $C_{i}$
of $F$ contains precisely one vertex $c_{i}$ which is an end-vertex of a path
which corresponds to an edge of $M_{K_{\mathrm{odd}}}.$ Let $u_{i}$ be a
neighbor of $c_{i}$ on $C_{i},$ for every $i=1,\ldots,2k.$ Further, let
$M_{\mathrm{odd}}$ be a maximum matching in $F_{\mathrm{odd}}$ which does not
cover vertices $u_{i},$ for every $i=1,\ldots,2k.$ Obviously,
$M=M_{\mathrm{odd}}\cup M_{\mathrm{even}}$ is a maximum matching in $F$ and
let $H=G-M$. Notice that vertices $u_{i}$, for $i=1,\ldots,2k,$ are the only
$3$-vertices of $H.$ Let $P_{i}^{\prime}=u_{2i-1}P_{i}u_{2i}$ be the path in
$H$ which connects vertices $u_{2i-1}$ and $u_{2i}$ which is obtained from
$P_{i}$ by adding edges $u_{2i-1}c_{2i-1}$ and $u_{2i}c_{2i}$ to it.
Obviously, $P_{i}^{\prime}$ is an $F$-path in $H$ connecting the pair
$u_{2i-1}$ and $u_{2i}$ of $3$-vertices of $H.$ Moreover, besides the edges
incident to its end-vertices, $P_{i}^{\prime}$ does not contain edges of odd
cycles from $F$ since $P_{i}$ does not contain them. Hence, $\mathcal{P=\{P}%
_{i}^{\prime}:i=1,\ldots,k\}$ is the desired $F$-matching in $H$.
\end{proof}

Notice that for a $2$-factor $F$ of an oddness two snark $G$ which has
precisely two odd cycles, $C_{1}$ and $C_{2},$ a graph $K_{\mathrm{odd}}$
corresponding to any perfect matching $M_{\mathrm{even}}$ has precisely two
vertices $C_{1}$ and $C_{2}$. Also, $K_{\mathrm{odd}}$ must contain the edge
$e=C_{1}C_{2},$ since $C_{1}$ and $C_{2}$ are odd length cycles in
$H_{\mathrm{odd}}.$ Thus, $M_{K_{\mathrm{odd}}}=\{e\}$ is a perfect matching
in $K_{\mathrm{odd}}$, so the above proposition yields the following corollary.

\begin{corollary}
\label{Cor_Foddness2}Let $G$ be an oddness two snark and $F$ a $2$-factor in
$G$ with precisely two odd cycles. Then there exists a maximum matching $M$ in
$F,$ such that $H=G-F$ has a simple $F$-matching.
\end{corollary}

Before we proceed, let us give several comments on Proposition
\ref{Tm_proplemP1}. First, for a snark $G$ and a $2$-factor $F$ in $G$ let $K$
be the graph obtained from $G$ by contracting the edges of $F.$ Notice that a
perfect matching in $K_{\mathrm{odd}}$ induces a $T$-join in $K$ in which
every odd degree equals $1.$ The converse does not seem to be the case.
Further, the condition of Proposition \ref{Tm_proplemP1} seems to be far from
necessary. Namely, in Proposition \ref{Tm_proplemP1} we construct a maximal
matching, i.e. a matching such that $H$ has precisely one $3$-vertex from each
odd cycle of $F.$ But essentially, more than one $3$-vertex can be created out
of every odd cycle of $F,$ they just must be conveniently placed on the cycle
of $F,$ so that there exists a matching on the cycle such that they are
connected to $3$-vertices created on other odd cycle and possibly even the
even cycles of $F.$. In other words, one could search for a $T$-join in
$K_{\mathrm{odd}}$ with additional properties.

\smallskip

\begin{figure}[h]
\begin{center}
\raisebox{-0.9\height}{\includegraphics[scale=0.6]{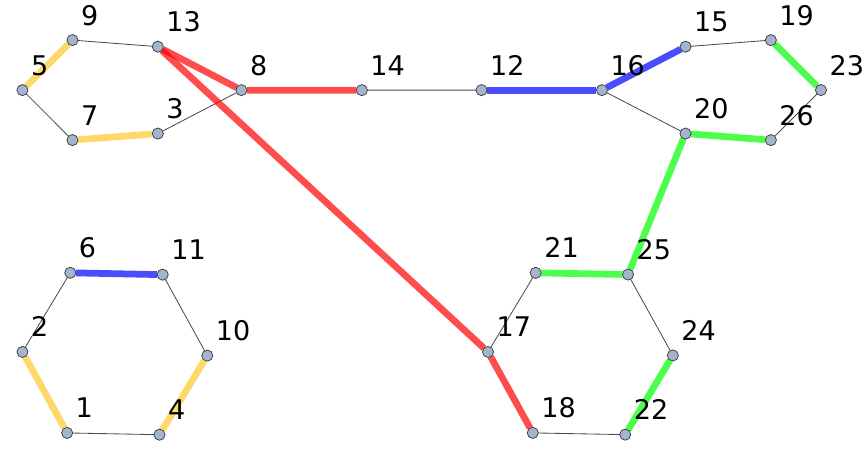}}
\end{center}
\caption{The figure shows the corrected graph $H^{\ast}$ of a graph $H$ from
Figure \ref{Fig08}. Notice that loops of $H$ are also loops of $H^{\ast}$ and
the two additional edges are introduced in $H,$ each of them is an additional
$F$-path so there exists an $F$-matching $\mathcal{P}^{\ast}$ in $H^{\ast}$
whose $F$-complement $\mathcal{L}^{\ast}$ is $3$-even.}%
\label{Fig11a}%
\end{figure}

Next, we provide a method of modifying a matching $M$ in a $2$-factor $F$ of a
general snark $G,$ in the case when $H=G-M$ does have an $F$-matching, but the
$F$-complement is not $3$-even, i.e. in the case of the problem P2). We will
motivate the method on the example of oddness two snark $G$ from Figure
\ref{Fig04} and the maximum matching $M$ in the $2$-factor $F$ from Figure
\ref{Fig07a}. For this choice of $M,$ the graph $H=G-M$ is shown in Figure
\ref{Fig08}. Notice again that the path $P$ connecting the pair of
$3$-vertices $r_{1}=8$ and $r_{2}=16$ is an $F$-path, hence $\mathcal{P}%
=\{P\}$ is an $F$-matching, but its $F$-complement $\mathcal{L}$ is not
$3$-even since it contains two loops with precisely one $3$-vertex of $H.$

In order to change the $3$-parity of the loop $L_{1}^{o}$ of $\mathcal{L}$
induced by vertices $\{3,5,7,8,9,13\}$ we need to create one more $3$-vertex
on that loop. This is done by modifying the maximum matching $M$ from Figure
\ref{Fig07a} into $M^{\ast},$ so that an additional vertex of $L_{1}^{o}$,
beside $r_{1}=8,$ is not covered by $M^{\ast}$ and thus becomes a $3$-vertex
in $H^{\ast}=G-M^{\ast}.$ Say we choose vertex $13$ to become additional
$3$-vertex, this implies we have to remove an edge $(13,17)$ from $M,$ but
then the vertex $17$ also becomes a $3$-vertex, see Figure \ref{Fig11a}. Now
the loop $L_{1}^{o}$ becomes $3$-even, since it contains the two $3$-vertices
$8$ and $13,$ but the loop induced by $\{17,18,21,22,24,25\}$ becomes $3$-odd,
since now it contains precisely one $3$-vertex and that is $17.$

Next, we wish to change the parity of this new $3$-odd loop, and we do that by
creating another $3$-vertex on it, say $25$. This is done by additionally
removing the edge $(25,20)$ from $M,$ which implies the vertex $20$ also
becomes a $3$-vertex beside the vertex $25,$ see again Figure \ref{Fig11a}.
Now the loop induced by $\{17,18,21,22,24,25\}$ becomes $3$-even, since it
contains precisely two $3$-vertices, these are $17$ and $25.$ Also, the other
additional $3$-vertex $20$ belongs to the loop $L_{2}^{o}$ which is $3$-odd
from the start and now it also becomes $3$-even. Hence, for the matching
$M^{\ast}=M\backslash\{(13,17),(25,20)\}$ the corresponding $F$-complement is
$3$-even.

\begin{figure}[h]
\begin{center}
\raisebox{-0.9\height}{\includegraphics[scale=0.12]{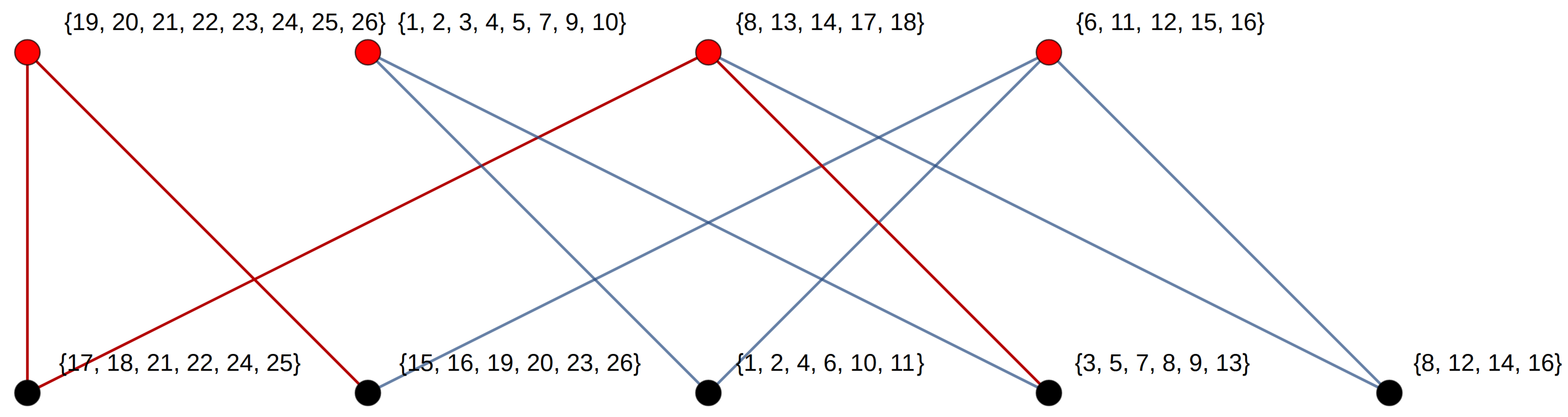}}
\end{center}
\caption{The figure shows the loop-cycle incidence graph $B$ for the snark $G$
as in Figure \ref{Fig04}, $2$-factor $F$ and maximal matching $M$ in $F$ as in
Figure \ref{Fig07a} and the graph $H$ as in Figure \ref{Fig08}. The upper row
of vertices belongs to $\mathcal{C}^{\ast}$ and the lower row of vertices
belongs to $\mathcal{L\cup\mathcal{P}}$, where vertices are labeled by the
vertex set of the corresponding cycle/path from $\mathcal{C}^{\ast}$ or a
loop/path of $\mathcal{L\cup\mathcal{P}}$. A path $P_{B}$ connecting the pair
of $3$-odd loops of $H$ is highlighted in red in the graph $B.$}%
\label{Fig10a}%
\end{figure}

Notice that in each step we created two additional $3$-vertices which both
belong to a same cycle of $F.$ In our example, these two vertices belonged to
two distinct loops in $H$ in both steps, hence we created a sequence of loops
$L_{0},L_{1},L_{2}$, where $L_{0}=L_{1}^{o}$ and $L_{2}=L_{2}^{o},$ such that
each two consecutive loops from the sequence contain a vertex of a same cycle
from $F.$ On the first and the last loop of the sequence one additional
$3$-vertex is created which changed their $3$-parity, while for the interior
loop of the sequence two additional $3$-vertices are created which preserves
the $3$-parity of the loop. All this indicates that we need to introduce the
notion of a loop-cycle-incidence graph.

\paragraph{Loop-cycle-incidence graph.}

Let $\mathcal{C}^{\ast}$ denote a set of all connected components of
$F-E_{\mathcal{P}}$ where $E_{\mathcal{P}}$ denotes the set of all end-edges
of paths from $\mathcal{P}$. Notice that $\mathcal{C}^{\ast}$ may contain only
cycles and paths. We define a \emph{loop-cycle-incidence graph} $B$ so that
$V(B)=\mathcal{C}^{\ast}\mathcal{\cup\mathcal{L\cup\mathcal{P}}}$ and a vertex
$L\in\mathcal{\mathcal{L\cup\mathcal{P}}}$ is adjacent to a vertex
$C\in\mathcal{C}^{\ast}$ if and only if $L$ and $C$ share an edge in $G$.
Notice that $B$ is a bipartite graph by definition and its bipartition sets
are $\mathcal{C}^{\ast}$ and $\mathcal{\mathcal{L\cup\mathcal{P}}}$.

For the kayak-paddle graph $H$ from Figure \ref{Fig08} and the corresponding
$2$-factor $F$ from Figure \ref{Fig07a}, the corresponding loop-cycle
incidence graph $B$ is illustrated by Figure \ref{Fig10a}. The modification of
the graph $H$ from Figure \ref{Fig08} into the graph $H^{\prime}$ from Figure
\ref{Fig11a}, that we described in previous text is done along the path
$P_{B}$ in $B$ (highlighted in red in Figure \ref{Fig10a}) which connects the
pair of $3$-odd loops $L_{1}^{o}$ and $L_{2}^{o}$ of $H.$ This is a very
simple example, and in the next theorem we provide a generalization of the
method to all snarks.

\begin{theorem}
\label{Prop_correctionM}Let $G$ be a snark, $F$ a $2$-factor in $G,$ and $M$ a
matching in $F.$ Let $\mathcal{P}$ be an $F$-matching in $H=G-M$, and $B$ the
corresponding loop-cycle incidence graph. If each component of $B-\mathcal{P}$
contains an even number of $3$-odd loops from $\mathcal{L\subseteq}V(B)$, then
$G$ has a proper $\mathbb{Z}_{4}\times\mathbb{Z}_{2}$-coloring.
\end{theorem}

\begin{proof}
Assume that $L_{1},L_{2},\ldots,L_{2q}$ are all $3$-odd loops in $\mathcal{L}$
denoted so that $L_{2i-1}$ and $L_{2i}$ belong to a same connected component
of $B-\mathcal{P}$, for $1\leq i\leq q$. This implies that there exists a path
$P_{i}$ in $B$ connecting $L_{2i-1}$ and $L_{2i}$ which does not contain
vertices of $\mathcal{P}\subseteq V(B).$

\medskip

\noindent\textbf{Claim A.}\emph{ We may assume that paths }$P_{i},$\emph{ for
}$i=1,\ldots,q,$\emph{ are pairwise edge disjoint.}

\smallskip

\noindent If paths $P_{i}$ are not pairwise edge disjoint, let $P$ be the
subgraph of $B$ induced by $\cup_{i=1}^{q}E(P_{i}),$ and let $P^{\prime}$ be
the subgraph of $P$ obtained from $P$ by removing all edges which belong to
even number of paths $P_{i}$ (i.e. $P^{\prime}$ is the symmetric difference of
all paths $P_{i}$). Then $P^{\prime}$ is an union of edge disjoint paths
$P_{i}^{\prime}$, for $1\leq i\leq q,$ each of which connects a pair of
$3$-odd loops of $\mathcal{L}$ and does not contain vertices of $\mathcal{P}%
\subseteq V(B).$ This establishes the claim.

\medskip\begin{figure}[h]
\begin{center}%
\begin{tabular}
[c]{ll}%
a) & \raisebox{-0.9\height}{\includegraphics[scale=0.53]{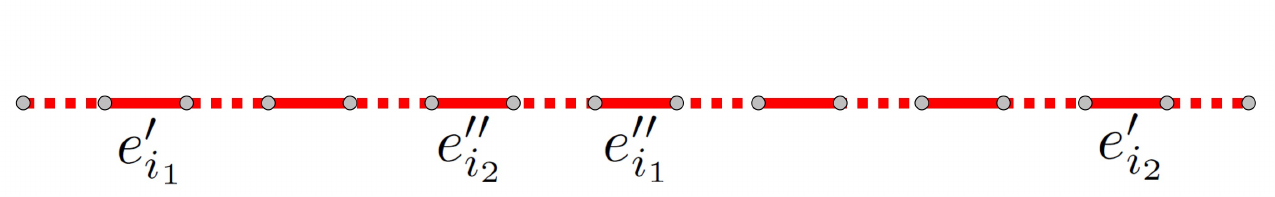}}\\
b) & \raisebox{-0.9\height}{\includegraphics[scale=0.53]{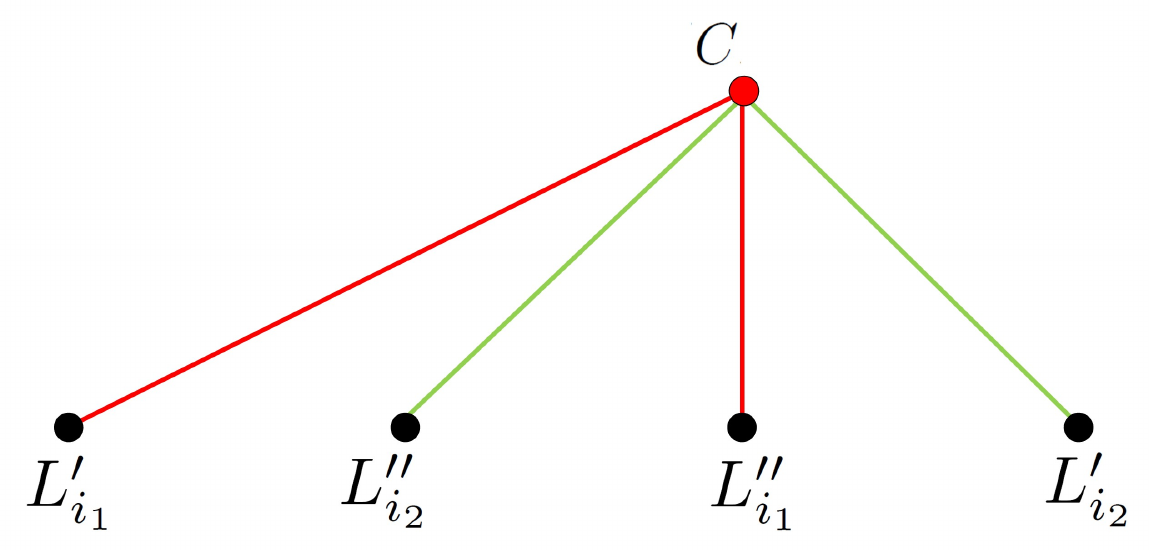}}
\end{tabular}
\end{center}
\caption{The figure shows a path $C\in\mathcal{C}_{P}^{\ast}$, where the edges
of $C$ which belong to $M$ are dashed. There are two paths in $B$ which go
through $C,$ the paths $P_{i_{1}}$ and $P_{i_{2}}.$ Hence, a) the edges
$e_{i_{1}}^{\prime}$ and $e_{i_{1}}^{\prime\prime}$ of $P_{i_{1}}$ and also
$e_{i_{2}}^{\prime}$ and $e_{i_{2}}^{\prime\prime}$ of $P_{i_{2}}$ are shown
on $C.$ In b) the corresponding segment $L_{i_{1}}^{\prime}CL_{i_{1}}%
^{\prime\prime}$ of $P_{i_{1}}$ is shown in red and the corresponding segment
$L_{i_{2}}^{\prime}CL_{i_{2}}^{\prime\prime}$ of $P_{i_{2}}$ is shown in
green.}%
\label{Figure06}%
\end{figure}

Let $\mathcal{C}_{P}^{\ast}\subseteq\mathcal{C}^{\ast}$ denote the set of all
cycles/paths of $\mathcal{C}^{\ast}$ which belong to at least one path
$P_{i}.$ Fix a cycle/path $C$ from $\mathcal{C}_{P}^{\ast}.$ Let $i_{1}\leq
i_{2}\leq\cdots\leq i_{r}$ be all indices such that $C$ is contained on
$P_{i_{k}}$. Let $L_{i_{k}}^{\prime}$ and $L_{i_{k}}^{\prime\prime}$ be the
pair of loops from a path $P_{i_{k}}$ incident to $C$ for every $1\leq k\leq
r.$ Since by Claim A the paths $P_{i_{k}}$ in $B$ are pairwise edge disjoint,
the loops of the set $\mathcal{L}_{C}=\{L_{i_{k}}^{\prime},L_{i_{k}}%
^{\prime\prime}:1\leq k\leq r\}$ are pairwise distinct, which implies they are
also pairwise edge disjoint. Denote by $e_{i_{k}}^{\prime}$ (resp. $e_{i_{k}%
}^{\prime\prime}$) an edge shared by $C$ and $L_{i_{k}}^{\prime}$ (resp.
$L_{i_{k}}^{\prime\prime}$), as illustrated by Figure \ref{Figure06}.a). Since
the loops of $\mathcal{L}_{C}$ are pairwise edge disjoint, all edges from the
set $E_{C}=\{e_{i_{k}}^{\prime},e_{i_{k}}^{\prime\prime}:1\leq k\leq r\}$ are
pairwise distinct, hence $\left\vert E_{C}\right\vert =2r.$

We say that two paths $P_{i_{k}}$ and $P_{i_{l}}$ are \emph{interlaced} on
$C\in\mathcal{C}_{P}^{\ast},$ if edges $e_{i_{k}}^{\prime}$ and $e_{i_{k}%
}^{\prime\prime}$ belong to distinct connected components of $C-\{e_{i_{j}%
}^{\prime},e_{i_{j}}^{\prime\prime}\}$ or vice versa. An example of two
interlaced paths $P_{i_{1}}$ and $P_{i_{2}}$ is shown in Figure \ref{Figure06}%
.a), and the corresponding edges of paths in $B$ are shown in Figure
\ref{Figure06}.b). In the following claim we show that paths $P_{i_{k}},$ for
$k=1,\ldots,r,$ can be rewired so that no two paths are interlaced on
$C$.\begin{figure}[h]
\begin{center}%
\begin{tabular}
[c]{ll}%
a) & \raisebox{-0.9\height}{\includegraphics[scale=0.53]{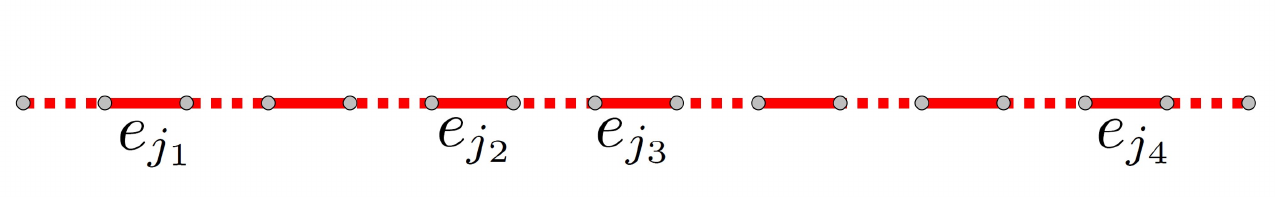}}\\
b) & \raisebox{-0.9\height}{\includegraphics[scale=0.53]{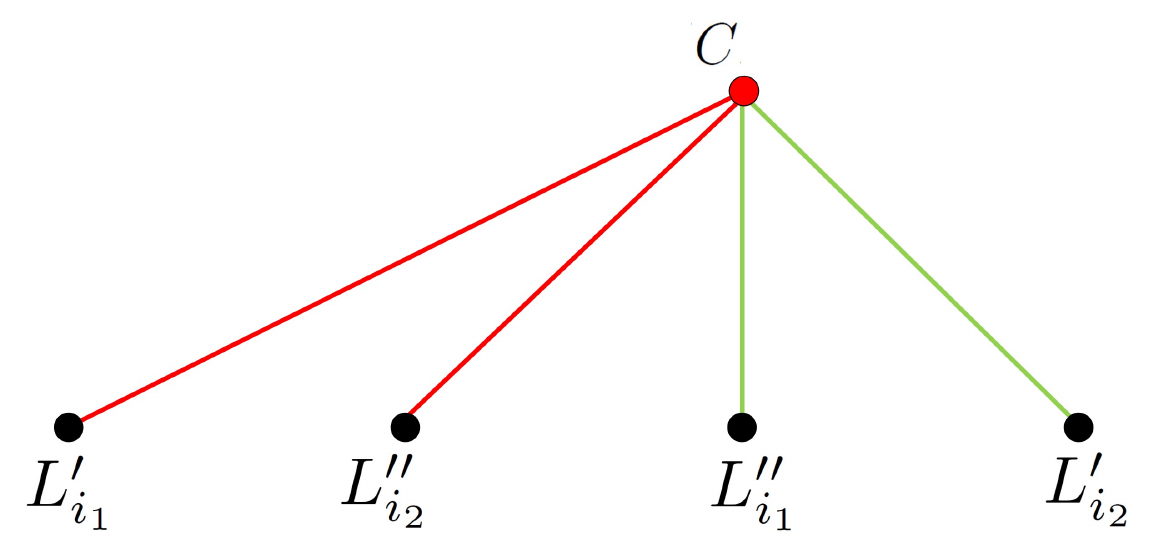}}
\end{tabular}
\end{center}
\caption{For the interlaced paths $P_{i_{1}}$ and $P_{i_{2}}$ from Figure
\ref{Figure06} and the corresponding edges $e_{i_{1}}^{\prime},$ $e_{i_{1}%
}^{\prime\prime},$ $e_{i_{2}}^{\prime}$ and $e_{i_{2}}^{\prime\prime},$ this
figure shows a) the relabeling of edges into $e_{j_{1}},$ $e_{j_{2}}$,
$e_{j_{3}}$ and $e_{j_{4}}$ according to the order od edges on $C$ and b)
rewiring of paths $P_{i_{1}}$ and $P_{i_{2}}$ into paths $P_{i_{1}}^{\prime}$
(in red) and $P_{i_{2}}^{\prime}$ (in green) which are not interlaced.}%
\label{Figure07}%
\end{figure}

\medskip

\noindent\textbf{Claim B.}\emph{ We may assume that paths }$P_{i_{k}},$\emph{
for }$k=1,\ldots,r,$\emph{ are pairwise not interlaced.}

\smallskip

\noindent Assume that vertices of the cycle/path $C\in\mathcal{C}_{P}^{\ast}$
are denoted so that $C=u_{0}u_{1}\cdots u_{g-1}u_{g}$ (here, $u_{g}=u_{0}$ if
$C$ is a cycle). Also, assume that edges of $C$ are denoted by $e_{j}%
=u_{j-1}u_{j}$ for $j=1,\ldots,g.$ Let $j_{1}<j_{2}<\cdots<j_{2r}$ be the
indices of the edges on $C$ such that $e_{j_{k}}\in E_{C}$ for $k=1,\ldots
,2r,$ i.e. for the edges of $E_{C}$ we take labels from $C$ so that they are
labeled according to their order on $C.$ This is illustrated by Figure
\ref{Figure07}.a). Let us now consider the first pair of edges $e_{j_{1}}$ and
$e_{j_{2}}$ from $E_{C}.$ According to the old notation we may assume that,
say, $e_{j_{1}}=e_{i_{k}}^{\prime}$ belongs to the loop $L_{i_{k}}^{\prime}$
of the path $P_{i_{k}}$ and $e_{j_{2}}=e_{i_{q}}^{\prime\prime}$ belongs to
the loop $L_{i_{q}}^{\prime\prime}$ of the path $P_{i_{q}}$. Now, we can
define a new path $P_{i_{1}}^{\prime}$ in $B$ as the one obtained from
$P_{i_{k}}$ and $P_{i_{q}}$ by making a connection $L_{i_{k}}^{\prime
}CL_{i_{q}}^{\prime\prime},$ see Figure \ref{Figure07}.b). Repeating the
procedure for the next pair of edges $e_{j_{3}}$ and $e_{j_{4}}$ from $E_{C},$
and so on until the last pair $e_{j_{2r-1}}$ and $e_{j_{2r}},$ yields paths
$P_{i_{k}}^{\prime}$ for $k=1,\ldots,r$ which are not interlaced, so the claim
is established.

\medskip

\begin{figure}[h]
\begin{center}%
\begin{tabular}
[c]{ll}%
a) & \raisebox{-0.9\height}{\includegraphics[scale=0.53]{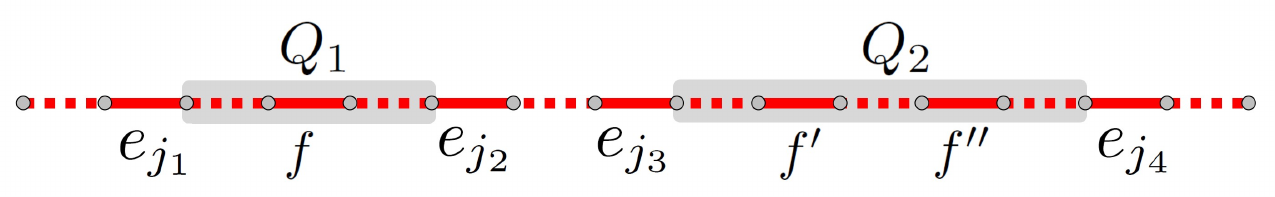}}\\
b) & \raisebox{-0.9\height}{\includegraphics[scale=0.53]{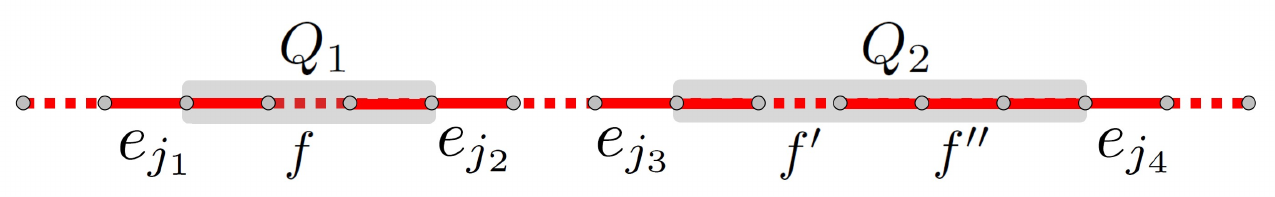}}
\end{tabular}
\end{center}
\caption{The figure shows a path $C\in\mathcal{C}_{P}^{\ast}$, where the edges
of $C$ which belong to $M$ are dashed. There are two non-interlaced paths in
$B$ which go through $C,$ the paths $P_{i_{1}}$ and $P_{i_{2}}.$ Hence, a) the
corresponding four edges $e_{j_{1}}$, $e_{j_{2}},$ $e_{j_{3}}$ and $e_{j_{4}}$
are shown on $C$ and the corresponding paths $Q_{1}$ and $Q_{2}$ connecting
them are shaded. In b) the modification of $M$ into $M^{\ast}$ along paths
$Q_{1}$ and $Q_{2}$ is shown, paths $Q_{i}$ contain edges $f,f^{\prime
},f^{\prime\prime}\in E(F)\backslash M$ where we assume that $f,f^{\prime}\in
E(\mathcal{P})$ and $f^{\prime\prime}\in E(\mathcal{L}).$}%
\label{Figure10}%
\end{figure}

Let us now consider the path $P_{i}$ connecting a pair of $3$-odd loops
$L_{2i-1}$ and $L_{2i}$ for some $i\in\{1,\ldots,q\}$ in $B-\mathcal{P}$. We
may assume $P_{i}=L_{0}^{i}C_{1}^{i}L_{1}^{i}\cdots C_{t}^{i}L_{t}^{i}$ where
$L_{0}^{i}=L_{2i-1}$ and $L_{t}^{i}=L_{2i}.$ Let us consider the segments of
length two of the path $P_{i},$ i.e. the subpaths $L_{j-1}^{i}C_{j}^{i}%
L_{j}^{i}$ of $P_{i}.$ Notice that each of the loops $L_{j-1}^{i}$ and
$L_{j}^{i}$ contains an edge from $C_{j}^{i}.$ Hence, by modifying the
matching $M$ on $C_{j}^{i}$ \textit{we wish to create two additional }%
$3$\textit{-vertices in }$H,$\textit{ one on the loop }$L_{j-1}^{i}$\textit{
and the other on the loop }$L_{j}^{i}.$ Doing this for each segment
$L_{j-1}^{i}C_{j}^{i}L_{j}^{i}$ of $P_{i}$ will result with an even number of
additional $3$-vertices on the interior loops of $P_{i}$ and an with odd
number of additional $3$-vertices on the two end-loops $L_{0}^{i}=L_{2i-1}$
and $L_{t}^{i}=L_{2i}$ of $P_{i}.$ Hence, the $3$-parity of interior loops of
$P_{i}$ will not change, and the $3$-parity of the two end-loops of $P_{i}$
will change. Since the end-loops of paths $P_{i}$ are $3$-odd loops of $H,$
this will convert all $3$-odd loops of $H$ into $3$-even loops in the modified
$H^{\ast},$ which is precisely what we wish.

Let us now return with our consideration to the segment $L_{j-1}^{i}C_{j}%
^{i}L_{j}^{i}$ of $P_{i}$ and our aim to create one additional $3$-vertex on
each of the loops $L_{j-1}^{i}$ and $L_{j}^{i}$. For simplicity, we may denote
$C=C_{j}^{i}$ and notice that $C\in\mathcal{C}_{P}^{\ast}$ since $C$ is
contained at least on the path $P_{i},$ but it may be contained on such a path
for several distinct values of $i.$ Recall that we denoted by $i_{1}\leq
i_{2}\leq\cdots\leq i_{r}$ all indices such that $C$ is contained on
$P_{i_{k}},$ and by $L_{i_{k}}^{\prime}$ and $L_{i_{k}}^{\prime\prime}$ the
pair of loops incident to $C$ on $P_{i_{k}}.$ Hence, $L_{i_{k}}^{\prime
}CL_{i_{k}}^{\prime\prime}$ is a segment of $P_{i_{k}}$ passing through $C$
for each $k=1,\ldots,r.$ By Claim B, we may assume that paths $P_{i_{k}}$ are
not interlaced on $C.$ For each of the loops $L_{i_{k}}^{\prime}$ and
$L_{i_{k}}^{\prime\prime}$ we chose precisely one their edge which belongs to
$C,$ and we denoted them by $e_{i_{k}}^{\prime}$ and $e_{i_{k}}^{\prime\prime
}$ respectively. Also, we denoted $E_{C}=\{e_{i_{k}}^{\prime},e_{i_{k}%
}^{\prime\prime}:1\leq k\leq r\}.$ Finally, we need the notion of an
$M$\emph{-path}, i.e. a subpath of $F$ which begins and ends by an edge of $M$
and its edges alternate between $M$ and $E(F)\backslash M.$

\medskip

\noindent\textbf{Claim C.}\emph{ For a segment }$L_{i_{k}}^{\prime}CL_{i_{k}%
}^{\prime\prime}$\emph{ of }$P_{i_{k}}$\emph{ there exists an }$M$\emph{-path
}$Q_{k}$\emph{ on }$C$\emph{ connecting edges }$e_{i_{k}}^{\prime}$\emph{ and
}$e_{i_{k}}^{\prime\prime}$\emph{ such that }$Q_{k}$\emph{ does not contain
edges of }$E_{C}.$

\smallskip

\noindent In order to prove the claim, we assume the vertex notation
$C=u_{0}u_{1}\cdots u_{g}$ (where $u_{g}=u_{0}$ if $C$ is a cycle) and edge
notation $e_{j}=u_{j-1}u_{j}.$ It follows that $e_{i_{k}}^{\prime}%
=e_{j_{2k-1}}$ and $e_{i_{k}}^{\prime\prime}=e_{j_{2k}}$ are two consecutive
edges of $E_{C}$ on $C.$ Let $Q_{k}$ be the subpath of $C$ connecting edges
$e_{i_{k}}^{\prime}=e_{j_{2k-1}}$ and $e_{i_{k}}^{\prime\prime}=e_{j_{2k}},$
see Figure \ref{Figure10}.a) for illustration. In other words, $Q_{k}$ is the
subpath of $C$ connecting the two selected edges from the two loops of the
segment $L_{i_{k}}^{\prime}CL_{i_{k}}^{\prime\prime}$ of the path $P_{i}.$
Notice that edges of $C$ alternate between $F$ and $M$, since $C\in
\mathcal{C}^{\ast}.$ Also, since $e_{i_{k}}^{\prime}=e_{j_{2k-1}}$ and
$e_{i_{k}}^{\prime\prime}=e_{j_{2k}}$ are the edges of $C$ which belong also
to loops $L_{i_{k}}^{\prime}$ and $L_{i_{k}}^{\prime\prime}$ of $H=G-M,$ this
implies they do not belong to $M.$ Since $C$ is a subgraph of $F,$ these two
edges must belong to $E(F)\backslash M.$ This implies that the path $Q_{k}$
connecting them on $C$ begins and ends by an edge of $M,$ hence $Q_{k}$ is an
$M$-path. Further, as we assumed that paths $P_{i_{k}}$ are not interlaced,
edges $e_{i_{k}}^{\prime}=e_{j_{2k-1}}$ and $e_{i_{k}}^{\prime\prime
}=e_{j_{2k}}$ are two consecutive edges of $E_{C}$ on $C,$ so $Q_{k}$ does not
contain an edge of $E_{C}$ and the clam is established.

\medskip

In view of the vertex notation on $C,$ notice that one end-vertex of $Q_{k}$
is $u_{j_{2k-1}}$ which is the second end-vertex of the $e_{j_{2k-1}},$ and
the other vertex of $Q_{k}$ is $u_{j_{2k}-1}$ which is the first end-vertex of
the edge $e_{j_{2k}}.$ Notice that $u_{j_{2k-1}}$ belongs to the loop
$L_{i_{k}}^{\prime}$ and $u_{j_{2k}-1}$ to the loop $L_{i_{k}}^{\prime\prime}$
of the segment $L_{i_{k}}^{\prime}CL_{i_{k}}^{\prime\prime}$ of the path
$P_{i}.$ In order to make $u_{j_{2k-1}}$ (resp. $u_{j_{2k}-1}$) to be an
additional $3$-vertex on the loop $L_{i_{k}}^{\prime}$ (resp. $L_{i_{k}%
}^{\prime\prime}$) we want to do the following:

\begin{itemize}
\item[$(a)$] remove the edge of $Q_{k}$ incident to $u_{j_{2k-1}}$ (resp.
$u_{j_{2k}-1}$) from $M$ which automatically includes it in $H.$
\end{itemize}

\noindent Thus, the vertex $u_{j_{2k-1}}$ (resp. $u_{j_{2k}-1}$) becomes a new
$3$-vertex in $H$ and if $Q_{k}$ is a path of length one we are done. This is
the case with the example from Figures \ref{Fig08} and \ref{Fig11a}. On the
other hand, if $Q_{k}$ is a longer path, then the neighbor of $u_{j_{2k-1}}$
(resp. of $u_{j_{2k}-1}$) on $Q_{k}$ also becomes additional $3$-vertex in the
modified $H.$ Let $e\in E(F)\backslash M$ be the edge of $Q_{k}$ incident to
the neighbor of $u_{j_{2k-1}}$ (resp. of $u_{j_{2k}-1}$), i.e. $e$ is the
second edge from the end of $Q_{k}.$ Notice the following:

\begin{itemize}
\item[$(b)$] if $e$ belongs to a loop $L\in\mathcal{L}$, then the $3$-parity
of $L$ changes;

\item[$(c)$] if $e$ belongs to a path $P\in\mathcal{P}$, then the existence of
an $F$-matching in the modified $H$ might come into question.
\end{itemize}

\noindent In case $(b)$ we remove both edges of $Q_{k}$ incident to $e$ from
$M,$ which means they are automatically included in $H.$ The inclusion of
these two edges into $H$ makes both end-vertices of $e$ to become $3$-vertices
in the modified $H,$ so the $3$-parity of $L$ does not change. In case $(c)$,
we also we remove both edges of $Q_{k}$ incident to $e$ from $M,$ but we
additionally remove $e$ from $H$ and include it in $M.$ Thus, we avoid
creating new $3$-vertices on the path $P$ which might disrupt its property of
being an $F$-path, but the removal of $e$ from $H$ rewires such paths. In the
sequel, we will establish that this rewiring preserves the property of
$F$-paths, so it suits us.

Let us define more precisely the modification of $M$ according to the previous
consideration. We modify the matching $M$ into $M^{\ast}$ on the cycle $C$ so
that outside paths $Q_{k}$ matchings $M$ and $M^{\ast}$ are the same, and
among edges which belong to a path $Q_{k}$ only edges which belong to a path
of $\mathcal{P}$ will be included into $M^{\ast}.$ Since edges of
$\mathcal{P}$ are not included in $M,$ this implies that generally $M^{\ast
}\not \subseteq M.$ To be more precise, let $Q_{C}$ be the union of all paths
$Q_{k}$ on $C$ for $1\leq k\leq r.$ Let $M_{1}^{\ast}(C)$ be the set of all
edges of $M\cap E(C)$ which do not belong to $Q_{C},$ i.e.
\[
M_{1}^{\ast}(C)=(M\cap E(C))\backslash E(Q_{C}).
\]
Further, let $M_{2}^{\ast}(C)$ be the set of all edges of $Q_{C}$ which belong
to a path in $\mathcal{P}$, i.e.
\[
M_{2}^{\ast}(C)=E(Q_{C})\cap E(\mathcal{P}).
\]
Hence, we define $M^{\ast}(C)=M_{1}^{\ast}(C)\cup M_{2}^{\ast}(C).$ The set
$M^{\ast}(C)$ is illustrated by Figure \ref{Figure10}.b). Next, we have to
join all $M^{\ast}(C)$ into one set $M^{\ast}$. Let $M_{0}$ be the set of all
edges of $M$ which are not contained on any $C\in\mathcal{C}_{P}^{\ast},$ then
we define
\[
M^{\ast}=M_{0}\cup\bigcup\nolimits_{C\in\mathcal{C}_{P}^{\ast}}M^{\ast}(C),
\]
and obviously $M^{\ast}\subseteq E(F).$ Let us now establish several
properties of $M^{\ast}$ and the graph $H^{\ast}=G-M^{\ast}.$

\medskip

\noindent\textbf{Claim D.} \emph{The set }$M^{\ast}$\emph{ is a matching in
}$F.$

\smallskip

\noindent Let us first establish that $M^{\ast}\subseteq E(F).$ To see this,
notice that $M_{0}\subseteq M\subseteq E(F).$ Notice further that for every
$C\in\mathcal{C}_{P}^{\ast}$ it holds that $M_{1}^{\ast}(C)\subseteq
E(C)\subseteq E(F)$ and $M_{2}^{\ast}(C)\subseteq E(Q_{C})\subseteq E(F).$
Hence $M^{\ast}\subseteq E(F).$

It remains to establish that $M^{\ast}$ is a matching. First, $M_{0}$ is a
matching in $F,$ since $M_{0}\subseteq M$ and $M$ is a matching in $F.$
Further, since $M_{0}$ contains edges which do not belong to any
$C\in\mathcal{C}_{P}^{\ast}$, an edge of $M_{0}$ cannot share an end-vertex
with an edge of $M^{\ast}(C)$ for every $C\in\mathcal{C}_{P}^{\ast}.$
Similarly, since any two cycles/paths $C,C^{\prime}\in\mathcal{C}_{P}^{\ast}$
are vertex disjoint, an edge of $M^{\ast}(C)$ cannot share an end-vertex with
an edge of $M^{\ast}(C^{\prime}).$

It remains to show that two edges $e,e^{\prime}\in M^{\ast}(C)$ cannot share
an end-vertex. Recall that $M^{\ast}(C)=M_{1}^{\ast}(C)\cup M_{2}^{\ast}(C),$
where $M_{1}^{\ast}(C)\cap E(Q_{C})=\emptyset$ and $M_{2}^{\ast}(C)\subseteq
E(Q_{C}).$ Hence, in the case of $e\in M_{1}^{\ast}(C)$ and $e^{\prime}\in
M_{2}^{\ast}(C),$ these two edges can share only an end-vertex of $Q_{C}.$
Yet, each path of $Q_{C}$ is an $M$-path, so its end-edges do not belong to
$\mathcal{P}$ and therefore they are not included in $M_{2}^{\ast}(C),$ which
implies $e$ and $e^{\prime}$ are vertex disjoint. Further, in case of
$e,e^{\prime}\in M_{1}^{\ast}(C)$ these two edges are vertex disjoint since
$M_{1}^{\ast}(C)\subseteq M.$ Finally, in case of $e,e^{\prime}\in M_{2}%
^{\ast}(C)=E(Q_{C})\cap E(\mathcal{P}),$ it follows that $e$ and $e^{\prime}$
belong to paths of $\mathcal{P}$ and they are contained in $E(F)\backslash M.$
Since edges of $F$ alternate with the edges of its complement on paths of
$\mathcal{P}$, it follows that $e$ and $e^{\prime}$ are vertex disjoint in
this case also, so the claim is established.

\medskip

\noindent\textbf{Claim E.} \emph{The }$3$\emph{-vertices of }$H$\emph{ are
also }$3$\emph{-vertices in }$H^{\ast}.$

\smallskip

\noindent Let $v$ be a $3$-vertex of $H$ and $e_{1}$ and $e_{2}$ the two edges
incident to $v$ in $F.$ It is sufficient to show that neither of $e_{1}$ and
$e_{2}$ is included in $M^{\ast}.$ Since $v$ is a $3$-vertex in $H,$ edges
$e_{1}$ and $e_{2}$ are not included in $M.$ Notice that $M_{0}\subseteq M$
and $M_{1}^{\ast}(C)\subseteq M$ for every $C\in\mathcal{C}_{P}^{\ast},$ so
edges $e_{1}$ and $e_{2}$ can be included in $M^{\ast}$ only if each of them
belongs to $M_{2}^{\ast}(C)$ for some $C\in\mathcal{C}_{P}^{\ast}.$

To see that neither of the edges $e_{1}$ and $e_{2}$ belongs to $M_{2}^{\ast
}(C),$ recall that $H$ has an $F$-matching $\mathcal{P}$. This means that the
vertex $v$ is an end-vertex of an $F$-path $P_{v}\in\mathcal{P}$. This further
implies that one of the edges $e_{1}$ and $e_{2},$ say $e_{1},$ is an end-edge
of $P_{v}.$ Therefore, the edge $e_{1}$ is not included in any $C\in
\mathcal{C}_{P}^{\ast}.$ Since $M_{2}^{\ast}(C)\subseteq E(Q_{C})\subseteq
E(C)$ for every $C\in\mathcal{C}_{P}^{\ast},$ it follows that $e_{1}$ is not
included in $M^{\ast}.$ As for the edge $e_{2},$ notice that $e_{2}$ must
belong to $\mathcal{L}$ since $e_{1}$ belongs to $\mathcal{P}$. Since
$M_{2}^{\ast}(C)\subseteq E(\mathcal{P})$, this implies that $e_{2}$ cannot
belong to $M_{2}^{\ast}(C)$ either. Hence, we have established that neither of
the two edges $e_{1}$ and $e_{2}$ incident to $v$ in $F$ is included in
$M^{\ast},$ which implies that $v$ is a $3$-vertex in $H^{\ast}$ and the claim
is established.

\medskip

\noindent\textbf{Claim F.} \emph{The subgraph of }$H$\emph{ induced by the
edges }$E(\mathcal{L})$\emph{ is also a subgraph of }$H^{\ast}.$

\smallskip

\noindent Let $L\in\mathcal{L}$ be a loop in $H.$ We have to show that each
edge of $L$ is included in $H^{\ast}$. If an edge of $L$ does not belong to
$F,$ then it is included in both $H$ and $H^{\ast},$ so we consider only edges
of $L$ which belong to $F.$ Let $e\in E(L)\cap E(F)$ be such an edge. Since
$e\in E(L)\subseteq E(H),$ it follows that $e$ is not included in $M.$ Let us
show that $e$ is not included in $M^{\ast}$ either. First, since
$M_{0}\subseteq M$ and $M_{1}^{\ast}(C)\subseteq M$ for every $C\in
\mathcal{C}_{P}^{\ast},$ the edge $e$ can belong to $M^{\ast}$ only if $e\in
M_{2}^{\ast}(C)$ for some $C\in\mathcal{C}_{P}^{\ast}.$ Notice that
$M_{2}^{\ast}(C)\subseteq E(\mathcal{P}),$ so the fact that $e\in
E(L)\subseteq E(\mathcal{L})$ together with $E(\mathcal{P})\cap E(\mathcal{L}%
)=\emptyset$ implies that $e$ does not belong to $M_{2}^{\ast}(C)$ either.
Hence, we have established that $e$ does not belong to $M^{\ast}$ which
implies that $e$ is included in $H^{\ast}.$ We conclude that the loop
$L\in\mathcal{L}$ remains preserved in $H^{\ast},$ and the claim is established.

\medskip

\noindent\textbf{Claim G.} \emph{All }$3$\emph{-vertices of }$H^{\ast}$\emph{
which are not }$3$\emph{-vertices of }$H$\emph{ belong to a subgraph of
}$H^{\ast}$\emph{ induced by }$E(\mathcal{L})$\emph{.}

\smallskip

\noindent Let $v$ be a $3$-vertex of $H^{\ast}$ which is not a $3$-vertex of
$H.$ Notice that a $3$-vertex of $H$ (resp. $H^{\ast}$) is determined by the
matching $M$ (resp. $M^{\ast}$). Since outside $Q_{C}$ the matchings $M$ and
$M^{\ast}$ are the same, it follows that $v$ must belong to $Q_{C}.$ Assume
first that $v$ is an end-vertex of a path from $Q_{C}.$ In this case, the
vertex $v$ does become a $3$-vertex of $H^{\ast}$ and we have to show that it
belongs to $\mathcal{L}$. Recall that $Q_{C}$ is defined as the collection of
paths on $C\in\mathcal{C}_{P}^{\ast}$ connecting the selected edges of loops
$L_{i_{k}}^{\prime}$ and $L_{i_{k}}^{\prime\prime}.$ Since $L_{i_{k}}^{\prime
}$ and $L_{i_{k}}^{\prime\prime}$ are the loops of $\mathcal{L}$ in $H,$ this
implies that an end-edge $v$ of a path of $Q_{k}$ belongs to $\mathcal{L}$.

Assume next that $v$ is an interior vertex of a path from $Q_{C}.$ Since paths
of $Q_{C}$ are $M$-paths, this means that edges of paths from $Q_{C}$
alternate between $M$ and $E(F)\backslash M.$ It follows that $v$ is an
end-vertex of an edge $e\in E(F)\backslash M.$ From $e\not \in M$ we conclude
that $e\in E(H),$ which means that $e$ belongs to either $\mathcal{L}$ or
$\mathcal{P}$. If $e\in\mathcal{L}$, then $v$ does become a $3$-vertex in
$H^{\ast}$, but $e\in\mathcal{L}$ implies $v\in\mathcal{L}$ and the claim
holds. On the other hand, if $e\in\mathcal{P}$ then $e$ is included in
$M^{\ast},$ so $v$ does not become a $3$-vertex in $H^{\ast}.$ Therefore, we
have established that every vertex of $Q_{C}$ which becomes a new $3$-vertex
in $H^{\ast}$ belongs to a loop of $\mathcal{L}$, so the claim is established.

\medskip

Since all $3$-vertices of $H$ belong to the loops from $\mathcal{L}$, Claims
E, F and G imply that all $3$-vertices of $H^{\ast}$ belong to the loops of
$\mathcal{L}$. Let $\mathcal{P}^{\ast}$ be a subgraph of $H^{\ast}$ induced by
all edges which do not belong to any loop of $\mathcal{L}$. From the fact that
all $3$-vertices of $H^{\ast}$ belong to $\mathcal{L}$, we conclude that that
$\mathcal{P}^{\ast}$ is a collection of paths whose end-vertices are
$3$-vertices of $H^{\ast}.$

\medskip

\noindent\textbf{Claim H.} \emph{The collection of paths }$\mathcal{P}^{\ast}%
$\emph{ is an }$F$\emph{-matching in }$H^{\ast}.$

\smallskip

\noindent The set $\mathcal{L}$ of loops in $H$ is an $F$-complement of an
$F$-matching $\mathcal{P}$ in $H,$ so edges of $F$ and $E(G)\backslash E(F)$
alternate on loops from $\mathcal{L}$. This implies that each loop of
$\mathcal{L}$ is of an even length. Since $\mathcal{P}^{\ast}$ is defined as
$H^{\ast}-E(\mathcal{L}),$ and for each $3$-vertex of $H^{\ast}$ two out of
the three edges incident to it must belong to $F,$ it follows that every path
of $\mathcal{P}^{\ast}$ begins and ends with an edge of $F.$ Hence, each path
of $\mathcal{P}^{\ast}$ is an $F\,$-path in $H^{\ast}$, so $\mathcal{P}^{\ast
}$ is an $F$-matching in $H^{\ast}$ which establishes the claim.

\medskip

Let us now consider the complement of $\mathcal{P}^{\ast}$ in $H^{\ast}$,
denote it by $\mathcal{L}^{\ast}.$ The following claim is now rather straightforward.

\medskip

\noindent\textbf{Claim I.} \emph{For the }$F$\emph{-complement }%
$\mathcal{L}^{\ast}$\emph{ of }$\mathcal{P}^{\ast}$\emph{ in }$H^{\ast}%
,$\emph{ it holds that }$\mathcal{L}^{\ast}=\mathcal{L}$\emph{.}

\smallskip

\noindent From $\mathcal{P}^{\ast}$ being defined as $H^{\ast}-E(\mathcal{L}%
),$ and the $F$-complement of $\mathcal{P}^{\ast}$ being defined as $H^{\ast
}-E(\mathcal{P}^{\ast}),$ it follows that $\mathcal{L}^{\ast}=\mathcal{L}$, as claimed.

\medskip

Our final claim is the following.

\medskip

\noindent\textbf{Claim J.} \emph{The }$F$\emph{-complement }$\mathcal{L}%
^{\ast}$\emph{ of }$\mathcal{P}^{\ast}$\emph{ in }$H^{\ast}$ \emph{is }%
$3$\emph{-even.}

\smallskip

\noindent According to Claim I, it holds that $\mathcal{L}^{\ast}=\mathcal{L}%
$, i.e. the $F$-complement of the $F$-matching in $H$ and $H^{\ast}$ is the
same. This means that we need to consider how the $3$-parity of loops from
$\mathcal{L}$ changes from $H$ and $H^{\ast}.$ In order for the claim to hold,
all $3$-odd loop from $\mathcal{L}$ must change $3$-parity from $H$ to
$H^{\ast},$ and all $3$-even loops of $\mathcal{L}$ must preserve $3$-parity
in $H^{\ast}.$ Claims E and F imply that all $3$-vertices which belong to a
loop $L\in\mathcal{L}$ in $H,$ also belong to a loop $L$ in $\mathcal{L}%
^{\ast}$ in $H^{\ast}.$ Hence, only additional $3$-vertices of $H^{\ast},$
which are not $3$-vertices in $H,$ may change the $3$-parity of a loop
$L\in\mathcal{L}$ from $H$ to $H^{\ast}.$

Since all additional $3$-vertices of $H^{\ast}$ belong to paths of $Q_{C}$ for
$C\in\mathcal{C}_{P}^{\ast},$ let us consider how such vertices influence the
parity of $3$-loops from $\mathcal{L}$. We first consider end-vertices of
paths in $Q_{C}.$ Let $v^{\prime}$ and $v^{\prime\prime}$ be two end-vertices
of a path from $Q_{C}$ for some $C\in\mathcal{C}_{P}^{\ast}.$ Recall that a
path of $Q_{C}$ is an $M$-path in $F$ which connects two selected edges
$e_{i_{k}}^{\prime}$ and $e_{i_{k}}^{\prime\prime}$ of the two loops
$L_{i_{k}}^{\prime}$ and $L_{i_{k}}^{\prime\prime}$ which form a segment
$L_{i_{k}}^{\prime}CL_{i_{k}}^{\prime\prime}$ of a path $P_{i_{k}}$ connecting
two $3$-odd loops $L_{2i-1}$ and $L_{2i}$ in $B.$ Hence, each end-vertex
$v^{\prime}$ and $v^{\prime\prime}$ of a path from $Q_{C}$ will change the
$3$-parity of a loop $L_{i_{k}}^{\prime}$ and $L_{i_{k}}^{\prime\prime}$ by
one. Given that each interior loop of $P_{i_{k}}$ belongs to precisely two
segments, the $3$-parity of interior loops of $P_{i_{k}}$ will remain the
same, and the $3$-parity of loops $L_{2i-1}$ and $L_{2i}$ which are the first
and the last loop in $P_{i_{k}}$ will change. Therefore, loops $L_{2i-1}$ and
$L_{2i}$ become $3$-even in $H^{\ast},$ since they were $3$-odd in $H,$ and
all other loops of $\mathcal{L}$ preserve $3$-parity.

Next, we need to consider how interior vertices of paths from $Q_{C}$
influence the $3$-parity of loops from $\mathcal{L}$, when they become
additional $3$-vertices in $H^{\ast}.$ Since edges of paths from $Q_{C}$
alternate between $E(F)\backslash M$ and $M,$ each interior vertex of $Q_{C}$
is an end-vertex of an edge $e\in E(F)\backslash M.$ If $e\in E(L)$ for some
$L\in\mathcal{\mathcal{L}}$, then both end-vertices of $e$ become additional
$3$-vertices in $H^{\ast},$ hence the $3$-parity of $L$ remains preserved in
$H^{\ast}.$ Otherwise, if $e\in E(P)$ for some $P\in\mathcal{P}$, then $e$ is
included in $M^{\ast}$, so the end-vertices of $e$ are not $3$-vertices in
$H^{\ast}.$ We conclude that the $3$-parity of $3$-odd loops changes from $H$
to $H^{\ast},$ and the $3$-parity of all other loops remains the same, so the
claim is established.

\medskip

Hence, we have constructed a matching $M^{\ast}$ in $F,$ so that the graph
$H^{\ast}=G-M^{\ast}$ has an $F$-matching $\mathcal{P}^{\ast}$ for which the
corresponding $F$-complement $\mathcal{L}^{\ast}$ is $3$-even. Finally, we can
now apply Theorem \ref{Lemma_karakterizacija} which yields that $G$ has a
proper $\mathbb{Z}_{4}\times\mathbb{Z}_{2}$-coloring.
\end{proof}

Notice that the above theorem does not apply to the case when there is a
connected component of $B-\mathcal{P}$ which contains an odd number of $3$-odd
loops of $\mathcal{L}$. To estimate how restricting this condition may be, let
us again consider the oddness two snark $G$ from Figure \ref{Fig04} and the
maximum matching $M$ in the $2$-factor $F$ from Figure \ref{Fig07a} which
result in the graph $H=G-M$ from Figure \ref{Fig08}. Let $P$ denote the only
path of $H$ connecting the pair of $3$-vertices $r_{1}=8$ and $r_{2}=16.$
Since $P$ is an $F$-path, the set $\mathcal{P=\{}P\mathcal{\}}$ is the
$F$-matching in $H.$ Hence, the set $\mathcal{P}$ contains only one element
which means that $B-\mathcal{P}$ is obtained from $B$ by removing only one
vertex, the vertex labeled by $\{8,12,14,16\}$ in Figure \ref{Fig10a}. Now,
the choice of $F$ and $M$ in $G$ would not be covered by Theorem
\ref{Prop_correctionM} only if every path connecting vertices $L_{1}^{o}$
labeled with $\{3,5,7,8,9,13\}$\ and $L_{2}^{o}$ labeled with
$\{15,16,19,20,23,26\}$ goes through the vertex $P$ labeled with
$\{8,12,14,16\}.$ In the case of this snark that is obviously not the case,
since there exists a path $P_{B}$ in $B,$ highlighted in red in Figure
\ref{Fig10a}, which connects $L_{1}^{o}$ and $L_{2}^{o}$ and does not contain
$P.$

Moreover, we could not find any maximum matching $M$ in a $2$-factor of any
oddness two snark $G$ that we considered, which does not satisfy the
conditions of Theorem \ref{Prop_correctionM}. Hence, we believe that snarks
$G$ with an $F$-matching in $H$ which are not covered by Theorem
\ref{Prop_correctionM} should be rare, not only among oddness two snarks, but
in general.

\section{Application to oddness two snarks}

Let us now apply the results of the previous two section to oddness two
snarks. Let $G$ be an oddness two snark and $F$ a $2$-factor in $G$ with
precisely two cycles. According to Theorem \ref{Lemma_karakterizacija}, we
have to find a matching $M$ in $F$ such that $H=G-F$ has an $F$-matching whose
$F$-complement in $3$-even. Corollary \ref{Cor_Foddness2} implies that there
exists a matching $M$ in $F$ such that $H$ has an $F$-matching $\mathcal{P}$.
If such a matching $M$ satisfies the assumptions of Theorem
\ref{Prop_correctionM}, then $G$ has a proper $\mathbb{Z}_{4}\times
\mathbb{Z}_{2}$-coloring, but $M$ may not satisfy these assumptions.
Nevertheless, using Theorem \ref{Prop_correctionM} we can obtain the partial
result for oddness two snarks as follows. We first need the following lemma.

\begin{lemma}
\label{Lemma_connected}Let $G$ be an oddness two snark, $F$ a $2$-factor in
$G$ with precisely two odd cycles and $M$ a maximum matching in $F.$ Then the
two $3$-odd loops of $H=G-M$ are connected in $B.$
\end{lemma}

\begin{proof}
Denote by $L_{1}^{o}$ and $L_{2}^{o}$ the pair of $3$-odd loops in $H.$ Assume
to the contrary that $L_{1}^{o}$ and $L_{2}^{o}$ are not connected in $B.$ Let
$\mathcal{L}_{1}$ denote the set of all loops of $H$ which are connected to
$L_{1}^{o}$ in $B.$ Notice that $L_{1}^{o}\in\mathcal{L}_{1}$ and $L_{2}%
^{o}\in\mathcal{L}\backslash\mathcal{L}_{1},$ so both $\mathcal{L}_{1}$ and
$\mathcal{L}\backslash\mathcal{L}_{1}$ are non-empty. Also, for any $L_{i}%
\in\mathcal{L}_{1}$ and $L_{j}\in\mathcal{L}\backslash\mathcal{L}_{1}$ it
holds that the set of edges $\{e=uv:u\in V(L_{i})$ and $v\in V(L_{j})\}$ is
empty, otherwise $L_{i}$ and $L_{j}$ would contain an edge from a same cycle
$C_{k}$ in $F,$ so they would be connected in $B,$ a contradiction. Denote
$V_{1}=\{u\in V(G):u\in\bigcup_{L\in\mathcal{L}_{1}}V(L)\}$ and notice that we
established that there is no edge in $G$ with one end-vertex in $V_{1}$ and
the other in $V(G)\backslash V_{1},$ hence $G$ is not connected, a contradiction.
\end{proof}

Now, we can establish the following partial result for oddness two snarks.

\begin{theorem}
Let $G$ be an oddness two snark. If $G$ contains a $2$-factor $F$ with
precisely two odd cycles and at most one even cycle, then $G$ has a proper
$\mathbb{Z}_{4}\times\mathbb{Z}_{2}$-coloring.
\end{theorem}

\begin{proof}
Let $M$ be any maximum matching in $F.$ If $H=G-M$ is a $\Theta$-graph, then
$G$ is colorable according to Proposition \ref{Lema_theta}. If $H=G-M$ is a
kayak-paddle graph, due to Corollary \ref{Cor_Foddness2} we may assume that
the path $P$ in the main component of $H$ is an $F$-path whose interior edges
do not belong to odd cycles of $F.$ If there is a path in $B$ connecting the
pair of $3$-odd loops $L_{1}^{o}$ and $L_{2}^{o}$ of $H$ which does not
contain $P,$ then $G$ is $\mathbb{Z}_{4}\times\mathbb{Z}_{2}$-colorable
according to Theorem \ref{Prop_correctionM}, so we may assume that such a path
does not exist, i.e. that all paths from $L_{1}^{o}$ to $L_{2}^{o}$ lead
through $P$ in $B.$

Let us show that the assumption that all paths from $L_{1}^{o}$ to $L_{2}^{o}$
in $B$ lead through $P$ results in contradiction. Denote by $C_{1}^{o}$ and
$C_{2}^{o}$ the two odd cycles of $F.$ Let $P_{B}=L_{0}C_{1}L_{1}C_{2}%
L_{2}\cdots C_{k}L_{k}$ be a shortest path in $B$ connecting $L_{0}=L_{1}^{o}$
and $L_{k}=L_{2}^{o}.$ Since $F$ contains at most three cycles, we conclude
that $k\leq3.$ Assume that $L_{i}=P$ and notice that $C_{i}\not =C_{i+1}$ must
hold, as otherwise the path $P_{B}$ would not be shortest. Since $F$ contains
at most three cycles, we conclude that either $C_{i}=C_{1}^{o}$ or
$C_{i+1}=C_{2}^{o}.$ If $C_{i}=C_{1}^{o},$ then the fact that $P$ does not
contain interior vertices from odd cycles of $F$ (and thus neither from
$C_{1}^{o}$) implies that $P$ contains a bridge of $G,$ a contradiction. The
argument for $C_{i+1}=C_{2}^{o}$ is analogous, and we are finished.
\end{proof}

Since permutation snarks have a $2$-factor with two odd cycles of a same
length and no even cycles, the above theorem immediately yields the following corollary.

\begin{corollary}
All permutation snarks have a proper $\mathbb{Z}_{4}\times\mathbb{Z}_{2}$-coloring.
\end{corollary}

\section{Concluding remarks}

In this paper we consider proper abelian colorings of cubic graphs. It is
known that there exist finite abelian groups which do not color all bridgeless
cubic graphs, namely cyclic groups of order smaller than 10, and that all
abelian groups of order at least 12 do. For the remaining abelian groups, the
so called exceptional groups $\mathbb{Z}_{4}\times\mathbb{Z}_{2}$,
$\mathbb{Z}_{3}\times\mathbb{Z}_{3}$, $\mathbb{Z}_{10}$ and $\mathbb{Z}_{11},$
it is known that the existence of a proper $\mathbb{Z}_{4}\times\mathbb{Z}%
_{2}$-coloring of a cubic graph $G$, implies the existence of a proper abelian
coloring of $G$ for all the remaining exceptional groups.

We first provide a characterization of $\mathbb{Z}_{4}\times\mathbb{Z}_{2}%
$-colorable graphs in terms of the existence of a matching $M$ in a $2$-factor
$F$ of a snark $G$ such that the graph $H=G-M$ has an $F$-matching whose
$F$-complement is $3$-even, see Theorem \ref{Lemma_karakterizacija}. This
characterization easily yields that all $3$-edge-colorable graphs are also
$\mathbb{Z}_{4}\times\mathbb{Z}_{2}$-colorable. Further, for oddness two
snarks $G$ it holds that any maximum matching $M$ in a $2$-factor $F$ of $G$
with precisely two odd cycles yields a graph $H=G-M$ whose main component is
either a $\Theta$-graph or a kayak-paddle graph. In the case when the main
component of $H$ is a $\Theta$-graph, the introduced characterization also
implies that $G$ has a proper $\mathbb{Z}_{4}\times\mathbb{Z}_{2}$-coloring.
In view of this characterization, Conjecture \ref{Con_exceptional} can be
restated as follows.

\begin{conjecture}
Every bridgeless cubic graph $G$ has a $2$-factor $F$ and a matching $M$ in
$F$ such that $H=G-M$ has an $F$-matching whose $F$-complement is $3$-even.
\end{conjecture}

Motivated by oddness two snarks for which the main component of $H$ is a
kayak-paddle graph, where for a matching $M$ an $F$-complement in $H$ may not
exist, or if it exists it may not be $3$-even, in the third section we
introduce a sufficient condition for the existence of a maximum matching $M$
in a $2$-factor $F$ of $G$ for which $H$ does have an $F$-complement, see
Theorem \ref{Tm_proplemP1}. This sufficient condition seems to be far from
necessary and it leaves space for improvements, so we propose the following problem.

\begin{problem}
For a snark $G,$ find a $2$-factor $F$ in $G$ and a matching $M$ in $F$ such
that $H=G-M$ has an $F$-matching.
\end{problem}

Further, assuming that for a matching $M$ in a $2$-factor $F$ of a snark $G$
the corresponding graph $H=G-M$ does have an $F$-matching, but such that the
$F$-complement is not $3$-even, in Theorem \ref{Prop_correctionM} we further
introduce a sufficient condition under which $M$ can be modified into a
matching $M^{\prime}$ for which the corresponding $F$-complement is $3$-even.
Theorem \ref{Prop_correctionM} applies to all matchings $M$ except in one very
particular case, which should be very rare and which we leave open as the
following problem.

\begin{problem}
Let $G$ be a snark, $M$ a matching in a $2$-factor $F$ of $G$ such that there
exists an $F$-matching $\mathcal{P}$ in $H=G-M,$ and let $B$ be the
corresponding loop-cycle-incidence graph. Show that the snark $G$ has a proper
$\mathbb{Z}_{4}\times\mathbb{Z}_{2}$-coloring in the case when there exists a
connected component of $B-\mathcal{P}$ with odd number of $3$-odd loops.
\end{problem}

Finally, it is worth to note that the characterization of Theorem
\ref{Lemma_karakterizacija} is related to the notion of edge reductions of a
snark which has already been considered in literature. Let $G$ be a cubic
graph and $e=uv$ an edge in $G.$ If $G^{\prime}$ is obtained from $G$ so that
the edge $e$ is removed from $G$ and the resulting $2$-vertices $u$ and $v$
are suppressed, we say that $G^{\prime}$ is obtained by an \emph{edge
reduction} of $G.$ In \cite{Stephen}, the snarks are classified according to
whether they can be reduced to a $3$-colorable graph by one edge reduction. In
\cite{Allie}, the number $e_{r}(G)$ of a snark $G$ is defined as the minimum
number of non-adjacent edges in $G$ such that the removal of these edges and
suppression of incident vertices renders a $3$-edge-colorable graph. There,
the following relation of $e_{r}(G)$ to the resistance $r(G)$ of $G$ is
conjectured, where the \emph{resistance} $r(G)$ is the minimum number of edges
which can be removed from $G$ so that the resulting graph is $3$-edge-colorable.

\begin{conjecture}
Let $G$ be a snark. Then $r(G)\geq2e_{r}(G)$.
\end{conjecture}

Several results related to $e_{r}(G)$ are reported in \cite{Allie3} and
\cite{Allie2}. Let $\omega(G)$ denote the oddness of a snark $G,$ the results
of this paper imply the following claim.

\begin{corollary}
If a cubic graph $G$ has a proper $\mathbb{Z}_{4}\times\mathbb{Z}_{2}%
$-coloring, then $2e_{r}(G)\leq\left\vert V(G)\right\vert -\omega(G).$
\end{corollary}

\begin{proof}
If a cubic graph $G$ is $\mathbb{Z}_{4}\times\mathbb{Z}_{2}$-colorable, then
Theorem \ref{Lemma_karakterizacija} implies that there exists a $2$-factor $F$
in $G$ and a matching $M$ in $F$ such that $H=G-M$ has an $F$-matching whose
$F$-complement is $3$-even. Since the $F$-complement is $3$-even, suppressing
of $2$-vertices of $H$ yields a $3$-colorable reduction of $G.$ It remains to
establish the upper bound on $M,$ and $M$ being a matching in a $2$-factor $F$
of $G$ implies $\left\vert M\right\vert \leq(\left\vert E(F)\right\vert
-\omega(G))/2.$ Since $F$ is a $2$-factor of $G,$ it follows that $\left\vert
E(F)\right\vert =\left\vert V(G)\right\vert $ which implies the claim.
\end{proof}

\bigskip

\bigskip\noindent\textbf{Acknowledgments.}~~Both authors acknowledge partial
support of the Slovenian research agency ARRS program\ P1-0383 and ARRS
project J1-3002. The first author also the support of Project
KK.01.1.1.02.0027, a project co-financed by the Croatian Government and the
European Union through the European Regional Development Fund - the
Competitiveness and Cohesion Operational Programme. The use of Wolphram Online
Cloud is also acknowledged for the creation of some figures.


\bigskip

\begin{thebibliography}{9}                                                                                                %


\bibitem {Allie}I. Allie, A note on reducing resistance in snarks, Quaest.
Math. 46(6) (2023), 1057--1067.

\bibitem {Allie3}I Allie, Oddness to resistance ratios in cubic graphs,
Discrete Math. 342 (2019), 387--392.

\bibitem {Allie2}I. Allie, E. M\'{a}\v{c}ajov\'{a}, M. \v{S}koviera, Snarks
with resistance $n$ and flow resistance $2n$, Electron. J. Combin. 29(1)
(2022), P1.44.

\bibitem {Archdeacon1986}D. Archdeacon, Coverings of graphs by cycles, Congr.
Numer. 53 (1986), 7--14.

\bibitem {Holroyd2004}F. Holroyd, M. \v{S}koviera, Colouring of cubic graphs
by Steiner triple systems, J. Combin. Theory Ser. B 91 (2004), 57--66.

\bibitem {Kral2009}D. Kr\'{a}l', E. M\'{a}\v{c}ajov\'{a}, O. Pangr\'{a}c, A.
Raspaud, J.-S. Sereni, M. \v{S}koviera, Projective, affine, and abelian
colorings of cubic graphs, European J. Combin. 30 (2009), 53--69.

\bibitem {Macajova2005}E. M\'{a}\v{c}ajov\'{a}, A. Raspaud, M. \v{S}koviera,
Abelian colourings of cubic graphs, Electron. Notes Discrete Math. 22 (2005), 333--339.

\bibitem {MacajovaFano}E. M\'{a}\v{c}ajov\'{a}, M. \v{S}koviera, Fano
colourings of cubic graphs and the Fulkerson Conjecture, Theoret. Comput. Sci.
349 (2005), 112--120.

\bibitem {Stephen}E. Steffen, Classifications and characterizations of snarks,
Discrete Math. 188 (1998), 183--203.
\end{thebibliography}
\end{document}